\documentclass[12pt]{amsart}
\usepackage{amscd,amsmath,amsthm,amssymb}
\usepackage{mathtools}
\usepackage[margin=2cm]{geometry}
\usepackage{epsfig}
\usepackage{rawfonts}
\usepackage{enumerate}
\usepackage{graphics}
\usepackage[ruled,vlined]{algorithm2e}
\usepackage{xspace}
\usepackage{graphicx}
\usepackage{mathrsfs}
\usepackage{amsmath}
\usepackage{amsfonts}
\usepackage{amssymb}
\usepackage{amsthm}
\usepackage{graphicx}
\usepackage{booktabs}
\usepackage{caption}
\usepackage{listings}
\usepackage{setspace}
\usepackage[mathscr]{eucal}
\usepackage{pgfplots}
\usepackage{hyperref}
\usepackage{wrapfig}
\usepackage{floatflt,epsfig}
\usepackage{ dsfont }
\usepackage{amscd}
\usepackage{tikz-cd}
\usepackage{fancyhdr}
\usepackage[all]{xy}
\usepackage{latexsym}
\usepackage{amscd}
\usepackage{pifont}
\usepackage{subfig}
\usepackage{easyReview}
\usepackage{subfig}
\usepackage{pstricks-add}
\usepackage{pgf,tikz,pgfplots}
\pgfplotsset{compat=1.15}
\usepackage{mathrsfs}
\usetikzlibrary{arrows}
\usetikzlibrary[patterns]
 \usepackage[normalem]{ulem}
 \usepackage{multicol}

\newcommand{\Dc}{\mathcal{D}}

\newcommand{\cB}{\mathcal{B}}
\newcommand{\cC}{\mathcal{C}}
\newcommand{\cD}{\mathcal{D}}
\newcommand{\Ec}{\mathcal{E}}

\newcommand{\cP}{\mathcal{P}}
\newcommand{\Pc}{\mathcal{P}}
\newcommand{\Qc}{\mathcal{Q}}

\newcommand{\Lc}{\mathcal{L}}

\newcommand{\cF}{\mathcal{F}}

\newcommand{\cR}{\mathcal{R}}
\newcommand{\Rc}{\mathcal{R}}

\newcommand{\rHP}{\mathrm{HP}}

\renewcommand{\qedsymbol}{$\square$}
\newcommand{\rev}{\mathrm{rev}}

\def\opn#1#2{\def#1{\operatorname{#2}}} % to make operators
\opn\chara{char} \opn\length{\ell} \opn\pd{pd} \opn\rk{rk}
\opn\projdim{proj\,dim} \opn\injdim{inj\,dim} \opn\rank{rank}
\opn\depth{depth} \opn\grade{grade} \opn\height{height}
\opn\embdim{emb\,dim} \opn\codim{codim}

\opn\Tr{Tr} \opn\bigrank{big\,rank}
\opn\superheight{superheight}\opn\lcm{lcm}
\opn\trdeg{tr\,deg}%\emph{
	\opn\reg{reg} \opn\lreg{lreg} \opn\ini{in} \opn\lpd{lpd}
	\opn\size{size} \opn\sdepth{sdepth}
	\opn\link{link}\opn\fdepth{fdepth}\opn\lex{lex}\opn\dist{dist}
	\opn\div{div} \opn\Div{Div} \opn\cl{cl} \opn\Cl{Cl}
	\opn\Spec{Spec} \opn\Supp{Supp} \opn\supp{supp} \opn\Sing{Sing}
	\opn\Ass{Ass} \opn\Min{Min}\opn\Mon{Mon}
	\opn\Ann{Ann} \opn\Rad{Rad} \opn\Soc{Soc}
	\opn\Im{Im} \opn\Ker{Ker} \opn\Coker{Coker} \opn\Am{Am}
	\opn\Hom{Hom} \opn\Tor{Tor} \opn\Ext{Ext} \opn\End{End}
	\opn\Aut{Aut} \opn\id{id}
	
	\opn\nat{nat}
	\opn\pff{pf}%   \pf exists already
	\opn\Pf{Pf} \opn\GL{GL} \opn\SL{SL} \opn\mod{mod} \opn\ord{ord}
	\opn\Gin{Gin} \opn\Hilb{Hilb}\opn\sort{sort}
	\opn\aff{aff} \opn
	\con{conv} \opn\relint{relint} \opn\st{st}
	\opn\lk{lk} \opn\cn{cn} \opn\core{core} \opn\vol{vol}
	\opn\link{link} \opn\star{star}\opn\lex{lex}\opn\set{set}
	\opn\gr{gr}
	
	\def\pot#1#2{#1[\kern-0.28ex[#2]\kern-0.28ex]}

	\opn\dirlim{\underrightarrow{\lim}}
	\opn\inivlim{\underleftarrow{\lim}}
	\let\tensor=\otimes

	\let\to=\rightarrow
	
	\def\Implies{\ifmmode\Longrightarrow \else
		\unskip${}\Longrightarrow{}$\ignorespaces\fi}
	\def\implies{\ifmmode\Rightarrow \else
		\unskip${}\Rightarrow{}$\ignorespaces\fi}
	\def\iff{\ifmmode\Longleftrightarrow \else
		\unskip${}\Longleftrightarrow{}$\ignorespaces\fi}
	\let\gets=\leftarrow

	\let\:=\colon
	\let\epsilon\varepsilon
	\let\kappa=\varkappa
	\def\qed{\ifhmode\textqed\fi
		\ifmmode\ifinner\quad\qedsymbol\else\dispqed\fi\fi}
	\def\textqed{\unskip\nobreak\penalty50
		\hskip2em\hbox{}\nobreak\hfil\qedsymbol
		\parfillskip=0pt \finalhyphendemerits=0}
	\def\dispqed{\rlap{\qquad\qedsymbol}}
	
	\opn\dis{dis}
	\def\pnt{{\raise0.5mm\hbox{\large\bf.}}}
	
	\opn\Lex{Lex}
	
        \newtheorem{Theorem}{Theorem}[section]
	\newtheorem{Lemma}[Theorem]{Lemma}
	\newtheorem{Corollary}[Theorem]{Corollary}
	\newtheorem{Proposition}[Theorem]{Proposition}
	\newtheorem{Remark}[Theorem]{Remark}
	
	\newtheorem{Example}[Theorem]{Example}

	\newtheorem{Conjecture}[Theorem]{Conjecture}

        \newtheorem{Notation}[Theorem]{Notation}
                \newtheorem{Assumption}[Theorem]{Assumption}
       \newtheorem{Setup}[Theorem]{Setup}

\begin{document}

        \title[Switching rook polynomial of collections of cells]{Switching rook polynomial of collections of cells}

     \author{Francesco Navarra}
    \address{Sabanci University, Faculty of Engineering and Natural Sciences, Orta Mahalle, Tuzla 34956, Istanbul, Turkey}
     \email{francesco.navarra@sabanciuniv.edu}
    
     \author{Ayesha Asloob Qureshi}      
     \address{Sabanci University, Faculty of Engineering and Natural Sciences, Orta Mahalle, Tuzla 34956, Istanbul, Turkey}	
     \email{aqureshi@sabanciuniv.edu, ayesha.asloob@sabanciuniv.edu}

      \author{Giancarlo Rinaldo}      
      \address{Department of Mathematics and Computer Sciences, Physics and Earth Sciences, University of Messina, Viale Ferdinando Stagno d'Alcontres 31, 98166 Messina, Italy}
      \email{giancarlo.rinaldo@unime.it}


    %%%%%%%%%%%%%%%%%%%%%%%%%%%%%%%%%%%%%%%%%%%%%%%%%%%%%%%%%%%%%%%%%%%%%%%%%%%%%%%%%%%%%%%%
    
	\thanks{The second and third authors are supported by Scientific and Technological Research Council of Turkey T\"UB\.{I}TAK under the Grant No: 124F113, and are thankful to T\"UB\.{I}TAK for their support. The first author is supported by Scientific and Technological Research Council of Turkey T\"UB\.{I}TAK under the Grant No: 122F128, and is thankful to T\"UB\.{I}TAK for their supports. The first and third authors are members of INDAM-GNSAGA and acknowledge its support.
The authors are grateful to the HPC group at Sabancı University for installing \texttt{Macaulay2} on the Tosun cluster and for providing access to such a valuable computational resource. Moreover, we would like to thank D. Grayson, M. Stillman, and D. Torrance for their personal communications and insightful suggestions.}

%%%%%%%%%%%%%%%%%%%%%%%%%%%%%%%%%%%%%%%%%%%%%%%%%%%%%%%%%%%%%%%%%%%%%%%%%%%%%%%%%%%%%%%%

\begin{abstract}
We explore the novel connection between rook placements on collections of cells, also known as pruned chessboards, and the algebraic properties of ideals generated by $2$-minors. We design an algorithm to compute the switching rook polynomial of a collection of cells and show that it coincides with the $h$-polynomial of the associated coordinate ring for all collections up to rank 10 and polyominoes up to rank 12. Motivated by this evidence, we conjecture that the correspondence holds in general, and we prove it for certain convex collections of cells by algebraic tools.  
\end{abstract}

%%%%%%%%%%%%%%%%%%%%%%%%%%%%%%%%%%%%%%%%%%%%%%%%%%%%%%%%%%%%%%%%%%%%%%%%%%%%%%%%%%%%%%%%%%%%%%%%%%%%%%%%%%%%%%

\maketitle

%%%%%%%%%%%%%%%%%%%%%%%%%%%%%%%%%%%%%%%%%%%%%%%%%%%%%%%%%%%%%%%%%%%%%%%%%%%%%%%%%%%%%%%%

\section*{Introduction}

Polyominoes are planar figures formed by joining unit squares edge to edge. Originating in recreational mathematics and classical combinatorics, they have been extensively studied in connection with tiling and enumeration problems; see \cite{G}. The study of polyomino ideals, initiated in \cite{Q}, has attracted considerable attention due to the rich interplay between the combinatorial structure of a polyomino and the algebraic properties of its associated binomial ideal. Given a field $K$ and a polyomino $\mathcal{P}$, the polyomino ideal $I_{\mathcal{P}}$ encodes the adjacency relations between the cells of $\mathcal{P}$, and its quotient ring $K[\mathcal{P}] = S_{\cP}/I_{\mathcal{P}}$ captures important algebraic invariants deeply influenced by the geometry of $\mathcal{P}$. Since their first appearance in 2012, polyomino ideals have been the subject of a broad and growing body of literature; see, for example, \cite[Chapter 8]{HHO_Book}, \cite{EHQR, JN, QRR, QSS, RR, KV}, and the references therein. Let $X$ be an $m \times n$ matrix of indeterminates. Polyomino ideals are generated by certain sets of $2$-minors of $X$, generalizing the notion of determinantal ideals $I_t(X)$ for $t=2$, and include several well-known classes of binomial ideals, such as those generated by $2$-minors of ladder shapes and the join-meet ideals of planar distributive lattices (toric ideal of Hibi rings). We refer reader to \cite{QRR} for more details. 

The initial motivation for connecting polyominoes with binomial ideals was to facilitate the study of ideals generated by $2$-minors and to translate their algebraic properties into combinatorial aspects of polyominoes. In recent years, however, a surprising and elegant connection has emerged between polyomino ideals and a classical object in enumerative combinatorics: the rook polynomial, originally introduced to count non-attacking rook placements on chessboards. For instance, the number of ways to place $k$ indistinguishable non-attacking rooks on an $m \times n$ rectangular chessboard $\mathcal{B}_{m,n}$ is captured by the well-known rook polynomial
\[
r_{\mathcal{B}_{m,n}}(t) = \sum_{k=0}^{n} \binom{m}{k} P(n,k) t^k,
\]

where $P(n,k) = n(n-1)\cdots(n-k+1)$. A convex polyomino can be viewed as a ``mutilated" or ``pruned" chessboard, possibly with corners removed. Studying rook polynomials in this broader setting offers further combinatorial insights; see \cite{AJ, GR, KR}, and the references therein. Interestingly, algebraic invariants of the polyomino ideal, such as the Hilbert series and the Castelnuovo–Mumford regularity of $K[\mathcal{P}]$, encode combinatorially meaningful data: namely, the switching rook polynomial (also referred to as the nested rook polynomial in \cite{AJ}) and the rook number of $\mathcal{P}$, respectively. This connection was first observed in \cite{EHQR}, where the authors showed that the Castelnuovo–Mumford regularity of $K[\mathcal{P}]$ coincides with the rook number of $\mathcal{P}$ when $\mathcal{P}$ is an L-convex polyomino. The connection was further explored in \cite{RR}, where it was observed that for simple thin polyominoes, the rook polynomial coincides with the $h$-polynomial of $K[\mathcal{P}]$, and the rook number again matches the Castelnuovo–Mumford regularity. Subsequently, %the second and third authors, 
Qureshi, Rinaldo and Romeo proposed the following conjecture:
\begin{Conjecture}
    \cite[Conjecture 3.2]{QRR}\label{conj:into} 
    If $\Pc$ is a simple polyomino (a polyomino without holes), then its switching rook polynomial coincides with the $h$-polynomial of the associated coordinate ring $K[\Pc]$.
\end{Conjecture}
Since then, several researchers have investigated various classes of polyominoes to further explore the relationship between rook polynomials and their algebraic counterparts. In this article, we significantly extend the previous conjecture to the more general setting of arbitrary collections of cells. We computationally verify that the switching rook polynomial of a collection of cells coincides with the $h$-polynomial of the associated coordinate ring for all collections up to rank 10 and for all polyominoes up to rank 12. Furthermore, we show that the conjecture holds for convex polyominoes $\mathcal{P}$ satisfying the conditions described in Theorem~\ref{Theorem: Qureshi condition to have Groebner basis}. Our results offer an algebraic framework for computing and understanding these combinatorial quantities for broad classes of convex polyominoes, including directed convex polyominoes, see Figure~\ref{fig:Ferrer diagram and stack}. In fact, we prove our main result, Theorem~\ref{Thm:main}, in a more general setting by considering convex collections of cells rather than polyominoes alone.

The paper is organized as follows. In Section~\ref{Section: Preliminaries}, we recall the necessary background on collections of cells and their associated ideals of $2$-minors, along with relevant definitions from commutative algebra.

Section~\ref{sec:2} provides a brief overview of the switching rook polynomial and related terminology, and then presents Theorem~\ref{thm: computational thm}, our main computational result. Complete details of the algorithm are given in \cite{N}, where the reader can find: \texttt{SageMath} scripts to generate polyominoes and collections of cells, \texttt{Macaulay2} code to compute the switching rook polynomial, the rook number, and the $h$-polynomial of the coordinate ring, and pre-generated datasets of collections of cells up to rank 10 and polyominoes up to rank 12 suitable for use with the code. Finally, based on these computations, we state Conjecture~\ref{conj}, which significantly extends \cite[Conjecture 3.2]{QRR} to arbitrary collections of cells, and also includes the correspondence between the rook number and the Castelnuovo–Mumford regularity of the coordinate ring.

In Section~\ref{sec:3}, we review the conditions introduced in \cite{Q} under which a collection of cells $\Pc$ admits a quadratic Gr\"obner basis. These are presented in Theorem~\ref{Theorem: Qureshi condition to have Groebner basis} as Condition~($\#$) and Condition~($\#'$).

In Section~\ref{Sec:4}, we prove our main result Theorem~\ref{Thm:main}: if a convex collection of cells $\mathcal{P}$ (in particular, any convex polyomino) satisfies either of Condition~($\#$) or Condition~($\#'$), then the switching rook polynomial of $\Pc$ coincides with the $h$-polynomial of $K[\Pc]$, the coordinate ring associated to $\mathcal{P}$. Moreover, the rook number of $\Pc$ equals the Castelnuovo–Mumford regularity of $K[\mathcal{P}]$. This framework allows rook polynomials and rook numbers to be computed using tools from computational algebra. For example, we demonstrate in Example~\ref{Exa: rook polynomial} how \texttt{Macaulay2}~\cite{M2} can be used to efficiently compute these invariants, providing a practical enumeration in more general settings. The proof of Theorem~\ref{Thm:main} relies on identifying a suitable symmetry of $\mathcal{P}$ that enables a meaningful dissection of the collection, allowing us to apply additive properties of Hilbert–Poincaré series.

%%%%%%%%%%%%%%%%%%%%%%%%%%%%%%%%%%%%%%%%%%%%%%%%%%%%%%%%%%%%%%%%%%%%%%%%%%%%%%%%%%%%%%%%%%%%%%%%%%%%%%%%%%

\section{Required terminologies related to collection of cells}\label{Section: Preliminaries}

This section is devoted to introducing the definitions and notations related to collections of cells and the associated ideals of $2$-minors.

Let $(i,j), (k,l) \in \mathbb{N}^2$. We define a partial order on $\mathbb{N}^2$ by setting $(i,j) \leq (k,l)$ if and only if $i \leq k$ and $j \leq l$. Given $a = (i,j)$ and $b = (k,l)$ in $\mathbb{N}^2$ with $a \leq b$, we define the set
\[
[a,b] = \{(m,n) \in \mathbb{N}^2 \mid i \leq m \leq k,\ j \leq n \leq l\}
\]
as an \emph{interval} in $\mathbb{N}^2$. If $i < k$ and $j < l$, then $[a,b]$ is called a \emph{proper} interval. In this case, we refer to $a$ and $b$ as the \emph{diagonal corners} of $[a,b]$, and define $c = (i,l)$ and $d = (k,j)$ as the \emph{anti-diagonal corners}. If $j = l$ (respectively, $i = k$), then $a$ and $b$ are said to be in a \emph{horizontal} (respectively, \emph{vertical}) position.

A proper interval $C = [a,b]$ with $b = a + (1,1)$ is called a \emph{cell} of $\mathbb{N}^2$. The points $a$, $b$, $c$, and $d$ are referred to as the \emph{lower-left}, \emph{upper-right}, \emph{upper-left}, and \emph{lower-right} corners of $C$, respectively. We denote the set of \emph{vertices} and \emph{edges} of $C$ by $V(C) = \{a, b, c, d\}, \quad E(C) = \{\{a,c\}, \{c,b\}, \{b,d\}, \{a,d\}\}.$ For a collection of cells $\cP$ in $\mathbb{N}^2$, the sets of vertices and edges are defined as $V(\cP) = \bigcup_{C \in \cP} V(C), \quad E(\cP) = \bigcup_{C \in \cP} E(C).$ The \emph{rank} of $\cP$, denoted by $|\cP|$, is the number of cells in $\cP$. By convention, the empty set is considered a collection of cells of rank $0$.

Consider two cells $A$ and $B$ in $\mathbb{N}^2$, with lower-left corners $a = (i,j)$ and $b = (k,l)$, respectively, and suppose that $a \leq b$. The \emph{cell interval} $[A,B]$, also referred to as a \emph{rectangle}, is the set of all cells in $\mathbb{N}^2$ whose lower-left corners $(r,s)$ satisfy $i \leq r \leq k$ and $j \leq s \leq l$.

Let $\Pc$ be a collection of cells. The cell interval $[A,B]$ is called the \emph{minimal bounding rectangle} of $\Pc$ if $\Pc \subseteq [A,B]$ and there exists no other rectangle in $\mathbb{N}^2$ that properly contains $\Pc$ and is properly contained in $[A,B]$ (see Figure~\ref{fig:Ferrer diagram and stack}). If the corners $(i,j)$ and $(k,l)$ are in horizontal (respectively, vertical) position, we say that the cells $A$ and $B$ are in \emph{horizontal} (respectively, \emph{vertical}) position. The cell interval $[A,B]$ is called a \textit{column} (respectively, a \textit{row}) of $\cP$ if the cells $A$ and $B$ are in vertical (respectively, horizontal) position, all the cells in $[A,B]$ belong to $\cP$, and there is no cell interval $[A',B']$ with $A'$ and $B'$ in vertical (respectively, horizontal) position such that $[A,B]\subsetneq [A',B']$. The collection of cells in Figure~\ref{Figure: Polyomino + weakly conn. coll. of cells + three disconnected comp.} (A) has thirteen columns and ten rows.

A finite collection of cells $\Pc$ is said to be \emph{weakly connected} if, for any two cells $C$ and $D$ in $\Pc$, there exists a sequence of cells $\mathcal{C} \colon C = C_1, \dots, C_m = D$ in $\Pc$ such that $V(C_i) \cap V(C_{i+1}) \neq \emptyset$ for all $i = 1, \dots, m-1$. For an illustration, see Figure~\ref{Figure: Polyomino + weakly conn. coll. of cells + three disconnected comp.} (B).

If $\Pc = \bigcup_{i=1}^s \Pc_i$, where each $\Pc_i$ is a weakly connected collection of cells and $V(\Pc_i) \cap V(\Pc_j) = \emptyset$ for all $i \neq j$, then $\Pc_1, \dots, \Pc_s$ are called the \emph{weakly connected components} of $\Pc$. The collection of cells in Figure~\ref{Figure: Polyomino + weakly conn. coll. of cells + three disconnected comp.} (C) has three weakly connected components.

A finite collection of cells $\Pc$ is called \emph{connected}, or simply a \emph{polyomino}, if for any two cells $C$ and $D$ in $\Pc$, there exists a sequence of cells $\mathcal{C} \colon C = C_1, \dots, C_m = D$ in $\Pc$ such that $C_i \cap C_{i+1}$ is an edge shared by both $C_i$ and $C_{i+1}$ for all $i = 1, \dots, m-1$. Such a sequence is called a \emph{path} from $C$ to $D$ in $\Pc$. An example of a polyomino is shown in Figure~\ref{Figure: Polyomino + weakly conn. coll. of cells + three disconnected comp.} (A).

A subcollection $\Pc' \subseteq \Pc$ is called a \emph{connected component} of $\Pc$ if $\Pc'$ is a polyomino and is maximal with respect to set inclusion; that is, for any $A \in \Pc \setminus \Pc'$, the union $\Pc' \cup \{A\}$ is not a polyomino. For instance, the collection of cells in Figure~\ref{Figure: Polyomino + weakly conn. coll. of cells + three disconnected comp.} (B) has two connected components $\Pc_1$ and $\Pc_2$.

\begin{figure}[h]
    \centering
    \subfloat[]{\includegraphics[scale=0.5]{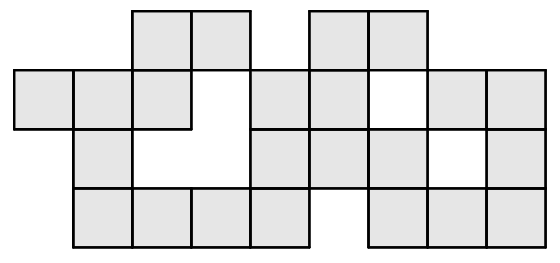}}
    \subfloat[]{\includegraphics[scale=0.5]{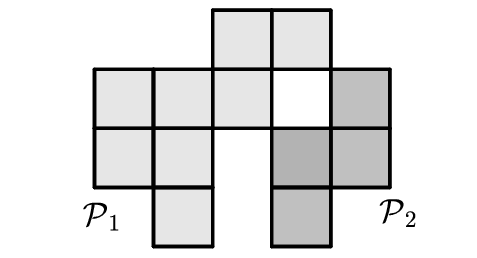}}
    \subfloat[]{\includegraphics[scale=0.5]{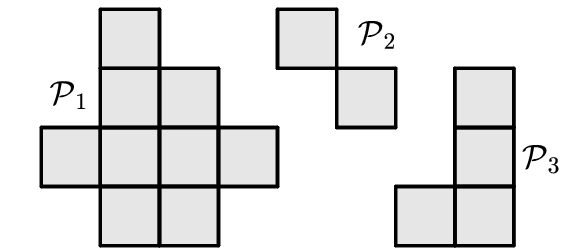}}
    \caption{A polyomino, a weakly connected collection of cells with two connected components, and a collection of cells with three weakly connected components.}
    \label{Figure: Polyomino + weakly conn. coll. of cells + three disconnected comp.}
\end{figure}

A collection of cells $\Pc$ is called \emph{row convex} (respectively, \emph{column convex}) if, for any two cells $A$ and $B$ of $\Pc$  in horizontal (respectively, vertical) position, the cell interval $[A, B]$ is entirely contained in $\Pc$. If $\Pc$ is both row and column convex, it is called \emph{convex}. In Figure~\ref{Figure: Polyomino + weakly conn. coll. of cells + three disconnected comp.} (C), the collection $\Pc = \Pc_1 \cup \Pc_2 \cup \Pc_3$ is column convex but not row convex. Moreover, each weakly connected component of $\Pc$ is convex.

A collection of cells $\Pc$ is said to be \emph{simple} if, for any two cells $C$ and $D$ not in $\Pc$, there exists a path of cells in the complement of $\Pc$ connecting $C$ to $D$. Roughly speaking, simple polyominoes are the polyominoes without holes. The collections in Figures~\ref{Figure: Polyomino + weakly conn. coll. of cells + three disconnected comp.} (A) and (B) are not simple, while the one in (C) is.

Let $\Pc$ be a convex polyomino with minimal bounding rectangle $[A,B]$. Then:
\begin{enumerate}
    \item $\Pc$ is called a \emph{Ferrer diagram} if at least three corner cells of $[A,B]$ are in $\Pc$ (Figure~\ref{fig:Ferrer diagram and stack} (A)).
    \item $\Pc$ is called a \emph{stack} if two adjacent corner cells of $[A,B]$ belong to $\Pc$ (Figure~\ref{fig:Ferrer diagram and stack} (B)).
    \item $\Pc$ is called a \emph{parallelogram} if two opposite corner cells of $[A,B]$ belong to $\Pc$ (Figure~\ref{fig:Ferrer diagram and stack} (C)).
    \item $\Pc$ is called \emph{directed convex} if at least one corner cell of $[A,B]$ belongs to $\Pc$ (Figure~\ref{fig:Ferrer diagram and stack} (D)).
\end{enumerate}

It follows directly from the above definitions that every Ferrer diagram, stack polyomino, and parallelogram polyomino is a directed convex polyomino.
\begin{figure}[h]
    \centering
    \subfloat[]{\includegraphics[scale=0.4]{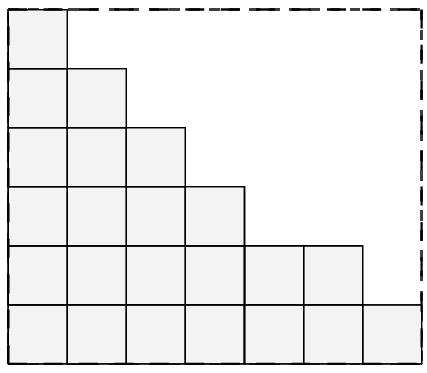}}\quad
    \subfloat[]{\includegraphics[scale=0.4]{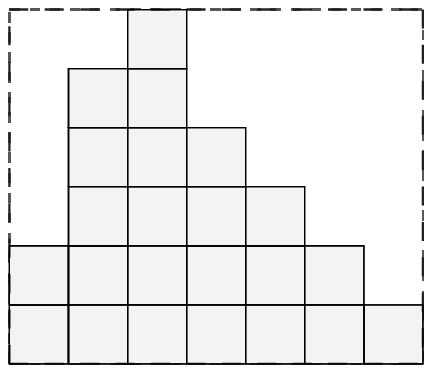}}\quad
    \subfloat[]{\includegraphics[scale=0.4]{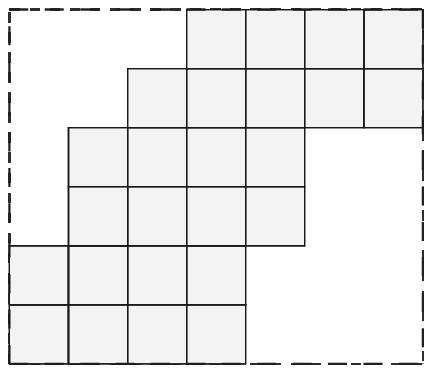}}\quad
    \subfloat[]{\includegraphics[scale=0.4]{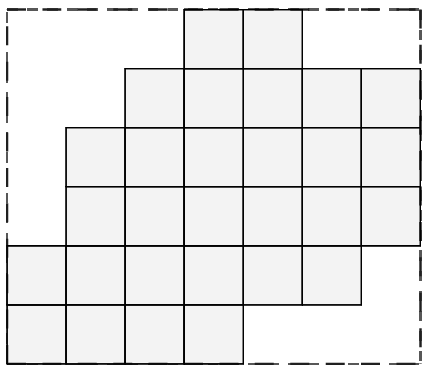}}
    \caption{From left to right: a Ferrer diagram, a stack, a parallelogram and a directed convex.}
    \label{fig:Ferrer diagram and stack}
\end{figure}

%%%%%%%%%%%%%%%%%%%%%%%%%%%%%%%%%%%%%%%%%%%%%%%%%%%%%%%%%%%%%%%%%%%%%%%%%%%%%%%%%%%%%%%%%%%%%%%%%%%%%%%%%%%%%%%%%%%%%%%%%%%%%%%%%%%%%%%%%%%%%%%%%%%%%%%%%%%%%%%

\subsection{Ideals of 2-minors Associated to Collections of Cells}

Let $\Pc$ be a collection of cells, and let $S_\Pc = K[x_v : v \in V(\Pc)]$ be the polynomial ring associated to $\Pc$, where $K$ is a field. A proper interval $[a,b]$ is called an \emph{inner interval} of $\Pc$ if all cells in $\Pc_{[a,b]}$ belong to $\Pc$. If $[a,b]$ is an inner interval of $\Pc$, with diagonal corners $a$ and $b$ and anti-diagonal corners $c$ and $d$, then the binomial $x_a x_b - x_c x_d$ is called an \emph{inner 2-minor} of $\Pc$. The ideal $I_{\Pc} \subset S_\Pc$ generated by all inner 2-minors of $\Pc$ is called the \emph{ideal of 2-minors} of $\Pc$. If $\Pc$ is a polyomino, then $I_{\Pc}$ is referred to as the \emph{polyomino ideal} of $\Pc$. The quotient ring $K[\Pc] = S_\Pc / I_{\Pc}$ is called the \emph{coordinate ring} of $\Pc$.

\begin{Remark}\label{rem:p=q}
Let $\Pc$ be a collection of cells, and let $\Qc$ be the collection obtained from $\Pc$ by a symmetry of the plane—i.e., by a translation, rotation, reflection, or glide reflection. Then $I_\Pc$ and $I_\Qc$ define the same ideal up to a relabeling of variables; in particular, $K[\Pc] \cong K[\Qc]$.
\end{Remark}

We conclude this subsection by introducing some notions and results related to the Hilbert-Poincaré series of a graded $K$-algebra $R/I$. Let $R$ be a graded $K$-algebra and $I$ a homogeneous ideal of $R$. Then $R/I$ inherits a natural grading given by $R/I = \bigoplus_{k \in \mathbb{N}} (R/I)_k.$ The function $\mathrm{H}_{R/I} \colon \mathbb{N} \to \mathbb{N}$ defined by $\mathrm{H}_{R/I}(k) = \dim_K (R/I)_k$ is called the \emph{Hilbert function} of $R/I$. The associated formal power series
\[
\mathrm{HP}_{R/I}(t) = \sum_{k \in \mathbb{N}} \mathrm{H}_{R/I}(k) t^k
\]
is called the \emph{Hilbert-Poincaré series} of $R/I$. By the Hilbert–Serre theorem, there exists a unique polynomial $h(t) \in \mathbb{Z}[t]$, called the \emph{$h$-polynomial} of $R/I$, such that $h(1) \neq 0$ and
\[
\mathrm{HP}_{R/I}(t) = \frac{h(t)}{(1 - t)^d},
\]
where $d$ is the Krull dimension of $R/I$. Moreover, if $R/I$ is Cohen–Macaulay, then the Castelnuovo–Mumford regularity satisfies $\mathrm{reg}(R/I) = \deg h(t)$. If $\Pc$ is a collection of cells, we denote the $h$-polynomial of $K[\Pc]$ by $h_{K[\Pc]}(t)$ throughout this paper.
%%%%%%%%%%%%%%%%%%%%%%%%%%%%%%%%%%%%%%%%%%%%%%%%%%%%%%%%%%%%%%%%%%%%%%%%%%%%%%%%%%%%%%%%%%%%%%%%%%%%%%%%%%%%%%%%%%%%%%%%%%%%%%%%%%%%%

\section{Switching rook polynomial of a collection of cells and a computational approach}\label{sec:2}

This section is devoted to a substantial extension of the algorithmic and computational framework introduced in \cite[Section 3]{QRR}, generalizing it to arbitrary collections of cells. We begin by recalling the basic definitions and notation regarding the classical combinatorics of rook placements on chessboard configurations. These concepts will be essential for interpreting algebraic invariants in terms of rook-theoretic quantities.

 Let $\cP$ be a collection of cells. Two rooks, $R_1$ and $R_2$, are said to be in \textit{attacking position} (or are called \textit{attacking rooks}) in $\cP$ if they are placed on cells belonging to the same row or the same column of $\cP$. %there exists a row or a column $[A,B]$ of $\cP$ such that $R_1$ and $R_2$ are placed in $A$ and $B$, respectively.
 On the other hand, two rooks are in \textit{non-attacking position} (or, they are \textit{non-attacking rooks}) in $\cP$ if they are not in attacking position. A \textit{$j$-rook configuration} in $\cP$ is a placement of $j$ rooks arranged in non-attacking positions in $\cP$, where $j> 0$; for convention, the $0$-rook configuration is $\emptyset$. Figure~\ref{Figura:example rook configuration} illustrates two examples of $5$-rook configurations.

\begin{figure}[h]
	\centering
    \subfloat[]{\includegraphics[scale=0.6]{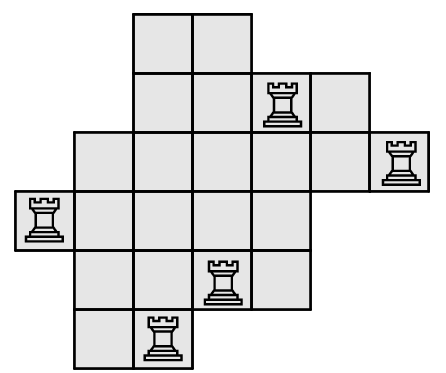}}\qquad\qquad
	\subfloat[]{\includegraphics[scale=0.6]{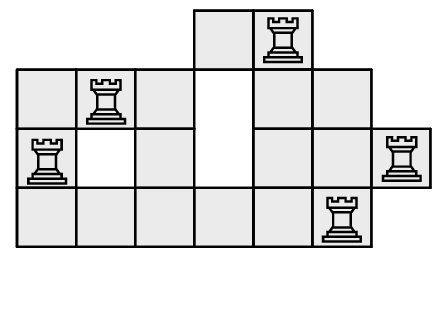}}
	\caption{Two examples of $5$-rook configurations in $\cP$.}
	\label{Figura:example rook configuration}
\end{figure}
 
 The \textit{rook number} $r(\cP)$ is the maximum number of non-attacking rooks that can be placed in $\cP$. For each $j \in [r(\mathcal{P})]$, we denote by $\mathcal{R}_j(\mathcal{P})$ the set of all $j$-rook configurations in $\mathcal{P}$, and by $r_j(\mathcal{P})$ the number of such configurations. By convention, we set $\mathcal{R}_0(\mathcal{P}) = \emptyset$ and $r_0(\mathcal{P}) := 1$. Moreover, let $\cR(\cP)=\cup_{k=0}^{r(\cP)}\cR_k(\cP)$. The \textit{rook polynomial} of $\cP$ is defined as 
\[
r_\cP(t) = \sum_{j=0}^{r(\cP)} r_j(\cP) t^j.
\]

 %Note that $\cR_0(\cP) \cup \cR_1(\cP) \cup \dots \cup \cR_{d}(\cP)$ forms a simplicial complex, called the \textit{rook complex} of $\cP$, where $d = r(\cP)$.
Two non-attacking rooks in $\cP$ are said to be in \textit{switching position}, or they are called \textit{switching rooks}, if they are placed in the diagonal (resp. anti-diagonal) cells of $I$, where $I$ is an inner interval of $\cP$. In such a case, we say the rooks are in diagonal (resp. anti-diagonal) position.

Fix $j \in \{0, \dots, r(\cP)\}$. Let $F \in \cR_j(\cP)$ and let $R_1$ and $R_2$ be two switching rooks of $F$ in diagonal (resp. anti-diagonal) position in $I$, where $I$ is an inner interval of $\cP$. Let $R_1'$ and $R_2'$ be the rooks in the corresponding anti-diagonal (resp. diagonal) cells of $I$. Then the set 
\[
(F \setminus \{R_1, R_2\}) \cup \{R_1', R_2'\}
\]
is also an element of $\cR_j(\cP)$. This operation is called a \textit{switch of $R_1$ and $R_2$}. 
This induces an equivalence relation $\sim$ on $\cR_j(\cP)$: for $F_1, F_2 \in \cR_j(\cP)$, we write $F_1 \sim F_2$ if $F_2$ can be obtained from $F_1$ through a sequence of switches. Let $\tilde{\cR}_j(\cP) = \cR_j(\cP)/\sim$ be the set of equivalence classes. We set $\tilde{r}_j(\cP) = |\tilde{\cR}_j(\cP)|$ for all $j \in [r(\cP)]$, and conventionally $\tilde{r}_0(\cP) = 1$.

The \textit{switching rook polynomial} of $\cP$ is the polynomial in $\mathbb{Z}[t]$ defined by
\[
\tilde{r}_{\cP}(t) = \sum_{j=0}^{r(\cP)} \tilde{r}_j(\cP) t^j.
\]

For instance, the rook polynomial of the polyomino $\cP$ in Figure~\ref{Figura:example rook configuration}(B) is $r_{\cP}(t) = 1 + 18 t + 115 t^2 + 337 t^3 + 484 t^4 + 339 t^5 + 108 t^6 + 12 t^7,$ and its switching rook polynomial is $\tilde{r}_{\cP}(t) = 1 + 18 t + 112 t^2 + 314 t^3 + 428 t^4 + 283 t^5 + 85 t^6 + 9 t^7.$

A collection of cells is called \textit{thin} if it does not contain any square tetromino. A \emph{square tetromino} is a $2 \times 2$ block of adjacent cells forming a square. Observe that if $\cP$ is a thin collection of cells, then $\tilde{r}_{\cP}(t) = r_{\cP}(t)$.

\begin{Remark}\label{rem:tensorprod}
Let $\cP$ be a collection of cells with weakly connected components $\cP_1, \cP_2, \ldots, \cP_s$. Then, we have $\tilde{r}_{\cP}(t) = \tilde{r}_{\cP_1}(t) \cdots \tilde{r}_{\cP_s}(t).$
Moreover, since $K[\cP] \cong K[\cP_1] \otimes \cdots \otimes K[\cP_s]$, we obtain $\rHP_{K[\cP]}(t) = \rHP_{K[\cP_1]}(t) \cdots \rHP_{K[\cP_s]}(t).$
\end{Remark}
Determining the rook number and rook polynomial of a polyomino, essentially a pruned chessboard, remains a highly challenging combinatorial problem. Investigating the connection with the associated polyomino ideal, along with the interplay between classical and switching rook polynomials, may offer a promising approach to the rook placement problem. 

Consider a rectangular polyomino, which can be viewed as a chessboard $\mathcal{B}_{m,n}$ of size $m \times n$. Its rook polynomial is well known and given by
\[
r_{\mathcal{B}_{m,n}}(t) = \sum_{k=0}^{n} \binom{m}{k} P(n,k) t^k,
\]
where $P(n,k) = n(n-1) \cdots (n-k+1)$. Moreover, an elegant relationship between the rook polynomial and the switching rook polynomial of a board is given in the following proposition.

\begin{Proposition}\label{prop:rook}
Let $\mathcal{B}_{m,n}$ be a rectangular polyomino of width $m$ and height $n$. Then, for every $0 \leq k \leq \min\{m,n\}$, we have
\[
r_k(\mathcal{B}_{m,n}) = k! \cdot \tilde{r}_k(\mathcal{B}_{m,n}).
\]
\end{Proposition} 

\begin{proof}
Let $k \in \{1, \dots, \min\{m,n\}\}$, and let $[\mathcal{T}] \in \tilde{\mathcal{R}}_k(\mathcal{B}_{m,n})$. Denote by $R_1, \dots, R_k$ and $C_1, \dots, C_k$ the $k$ rows and columns of $\mathcal{B}_{m,n}$ on which the rooks of $\mathcal{R}$ are placed. 

Set $\mathcal{S} = (\cup_{i \in [k]} R_i) \cap (\cup_{j \in [k]} C_j)$, which can be viewed as a $k \times k$ square. Label the rows and columns of $\mathcal{S}$ increasingly from $1$ to $k$, left to right for columns and bottom to top for rows. Each cell is then identified by a pair $(i,j)$ in $[k] \times [k]$, representing its row and column, respectively.

In this setup, each configuration of $k$ non-attacking rooks corresponds to a permutation $\sigma: [k] \to [k]$, and vice versa. We identify the identity permutation $\mathrm{id}(i) = i$ with the $k$-rook configuration $\mathcal{T}$. Since the equivalence class $[\mathcal{T}]$ contains exactly $k!$ elements, thus
\[
r_k(\mathcal{B}_{m,n}) = \sum_{[\mathcal{T}] \in \tilde{\mathcal{R}}_k(\mathcal{B}_{m,n})} |[\mathcal{T}]| = \tilde{r}_k(\mathcal{B}_{m,n}) \cdot k!.
\]
\end{proof}

For example, for the board $\mathcal{B}_{3,4}$, we have
\[
r(\mathcal{B}_{3,4})(t) = 1 + 12t + 36t^2 + 24t^3 \quad \text{and} \quad \tilde{r}(\mathcal{B}_{3,4})(t) = 1 + 12t + 18t^2 + 4t^3.
\]

It follows from the proof of Proposition~\ref{prop:rook} that for any weakly connected collection of cells $\mathcal{P}$, we have
\[
\tilde{r}_i(\mathcal{P}) \leq r_i(\mathcal{P}) \leq i! \cdot \tilde{r}_i(\mathcal{P}),
\]
for every $0 \leq i \leq r(\mathcal{P})$.

We now describe the computational approach used to compute the switching rook polynomial and the rook number of a collection of cells, which underlies the following result.

\begin{Theorem}\label{thm: computational thm}
Let $\cP$ be a collection of cells. Then $h_{K[\cP]}(t)$ coincides with the switching rook polynomial of $\cP$, and $\mathrm{reg}(K[\cP])$ equals the rook number of $\cP$ in the following cases:
\begin{itemize}
    \item when $\cP$ is a collection of cells of rank at most $10$;
    \item when $\cP$ is a polyomino of rank at most $12$.
\end{itemize}
\end{Theorem}

 \begin{proof}
To establish this result, we implemented some routines in \texttt{SageMath} \cite{sage} and \texttt{Macaulay2} \cite{M2}, shared in \cite{N}, which, for a fixed rank $n$, performs the following steps:

1) Compute the set $L$ of all collections of cells of rank $n$ (using \texttt{SageMath}). The corresponding datasets are provided as plain-text files containing all collections of cells and polyominoes that we tested. Tables \ref{tab:collection} and \ref{tab:polyominoes} summarize the number of weakly connected collections of cells and polyominoes, respectively.

\begin{table}[h!]
\centering
\begin{tabular}{|>{\centering\arraybackslash}p{6cm}|c|c|c|c|c|c|c|c|c|}
\hline
Rank & 2 & 3 & 4 & 5 & 6 & 7 & 8 & 9 & 10\\
\hline
\parbox[c]{6cm}{\centering\vspace{1mm} Number weakly connected\\collections of cells\vspace{1mm}} & 2 & 5 & 22 & 94 & 524 & 3031 & 18770  & 118133 & 758381 \\
\hline
\end{tabular}
\caption{Number of weakly connected collections of cells up to symmetries.}
\label{tab:collection}
\end{table}

\begin{table}[h!]
\centering
\begin{tabular}{|c|c|c|c|c|c|c|c|c|c|c|c|}
\hline
Rank & 2 & 3 & 4 & 5 & 6 & 7 & 8 & 9 & 10 & 11 &12  \\
\hline
Number of polyominoes & 1 & 2 & 5 & 12 & 35 & 108 & 369 & 1285 & 4655 & 17073 & 63600 \\
\hline
\end{tabular}
\caption{Number of polyominoes up to symmetries.}
\label{tab:polyominoes}
\end{table}

2) We present an algorithm, implemented in \texttt{Macaulay2}, for computing the switching rook polynomial of a collection of cells (see Section~2, file \texttt{RookPol.m2} in \cite{N}). A collection is encoded as a list of lists, where each cell is represented by the pair of lists corresponding to its diagonal corners. For instance, the square tetromino is encoded as
\[
\texttt{\{\{\{1,1\},\{2,2\}\},\{\{2,1\},\{3,2\}\},\{\{1,2\},\{2,3\}\},\{\{2,2\},\{3,3\}\}\}}.
\]

\medskip
Let $Q$ be a collection of cells. The first step consists in computing the set of all non-attacking rook configurations in $Q$ (see Algorithm \ref{alg:allconf}); here we need the auxiliary function \texttt{isNonAttackingRooks(A,B,Q)}, which returns \texttt{true} if two cells $A,B\in\mathbb{Z}^2$, where two rooks are placed, are in horizontal or vertical position in $\mathbb{Z}^2$ and, in addition, belong to an inner interval of $Q$.  

\begin{algorithm}[h]
\DontPrintSemicolon
\KwData{A collection of cells $Q$.}
\KwResult{All non-attacking rook configurations on $Q$, grouped by cardinality.}

Sort $Q$.\;
Initialize $\texttt{conf} = \{\{c\}\mid c\in Q\}$ (all $1$-rook configurations).\;
Initialize $\texttt{AllConf} = \{\texttt{conf}\}$.\;

\While{true}{
    $\texttt{out}=\emptyset$.\;
    \For{$N\in$ \textnormal{\texttt{conf}}}{
        \For{$c\in Q$}{
            \If{\textnormal{\texttt{all(N, n-> isNonAttackingRooks(c, n, Q)):}} $c$ is in non-attacking position with any rooks in $N$}{
                $FF=\texttt{sort}(N\cup\{c\})$.\;
                $\texttt{out}=\texttt{out}\cup\{FF\}$.\;
            }
        }
    }
    Remove duplicates: $\texttt{out}=\texttt{set}(\texttt{out})$.\;
    \If{$\# \textnormal{\texttt{out}}=0$}{\Return $\texttt{AllConf}$.\;}
    Append level: $\texttt{AllConf}=\texttt{AllConf}\cup\{\texttt{out}\}$.\;
    Update current layer: $\texttt{conf}=\texttt{out}$.\;
}
\caption{Computation of all non-attacking rook configurations on $Q$}
\label{alg:allconf}
\end{algorithm}

The next step, where we compute the coefficients of the switching rook polynomial (see Algorithm \ref{alg:switching}), requires the \texttt{Macaulay2} package \texttt{Graphs} and two following two short auxiliary functions.  
\begin{itemize}
    \item \texttt{SwitchOperation(A,B,Q)} takes two rooks $A,B\in Q$ and checks whether they are in switching position; if so, it returns the switched configuration;  
    \item By using the previous one, \texttt{areSwitchEquivalent(R1,R2,Q)} tests whether two rook configurations $R_1$ and $R_2$ can be obtained from one another by a single switch operation.
\end{itemize}

\begin{algorithm}[h]
\DontPrintSemicolon
\KwData{A collection of cells $Q$.}
\KwResult{The coefficients list $\{1,c_1,c_2,\ldots,c_r\}$ of the switching rook polynomial of $Q$.}

$\textnormal{\texttt{RookConf}} \gets \texttt{AllNonAttackingRookConfigurations}(Q)$\; $\textnormal{\texttt{RookNumber}} \gets |\textnormal{\texttt{RookConf}}|$ \;
Initialize $\texttt{coeffRookPol} =\{1\}$\;
\For{$k \in \{0,\ldots,\textnormal{\texttt{RookNumber}}-1\}$}{
    $\textnormal{\texttt{kRookSet}} \gets \textnormal{\texttt{RookConf}}\# k$ \;
    Build a graph $G_k$ with:
    \begin{enumerate}
       \item vertices = elements of $\textnormal{\texttt{kRookSet}}$,
       \item edges between $R_1,R_2$ if \texttt{areSwitchEquivalent(R1,R2,Q)==true}.
    \end{enumerate}
    \texttt{kCoeff}$\gets$ number of connected components of $G_k$\;
    $\texttt{coeffRookPol} = \texttt{coeffRookPol} \cup \{\texttt{kCoeff}\}$\;
}
\Return $\texttt{coeffRookPol}$\;

\caption{Algorithm for computing the switching rook polynomial of $Q$}
\label{alg:switching}
\end{algorithm}

 3) Let $L$ be the list of all collections of cells of rank $n$. We finally provide the function \texttt{TestConj(L)} to test the conjectures for all collections of cells in $L$. Specifically, for each collection $Q \in L$:  
\begin{itemize}
    \item compute the inner $2$-minor ideal of $Q$ (see Section 1, file \texttt{RookPol.m2} in \cite{N});
    \item compute the $h$-polynomial and the regularity of $K[Q]$;  
    \item compute the switching rook polynomial and the rook number of $Q$;  
    \item compare the two polynomials and check whether $\mathrm{reg}(K[Q])$ equal the rook number of $Q$.  
\end{itemize}
\item The output is “Conjecture is ok” if both equalities hold; otherwise return counterexamples.  

In all tested cases, that is, $n=2,\dots,10$ for collections of cells and $n= 11, 12$ for polyominoes, the program confirmed the conjectured equalities. We point out that the cases $n = 9, 10$ for collections of cells and $n = 12$ for polyominoes were verified using the HPC cluster Tosun provided by Sabanci University, due to their large computational complexity, which made it impossible to verify them on a standard machine.
\end{proof}

The previous result leads us to formulate the following.

\begin{Conjecture}\label{conj}
Let $\cP$ be a collection of cells. Then the switching rook polynomial of $\cP$ coincides with the $h$-polynomial of $K[\cP]$, and the rook number of $\cP$ equals the regularity of $K[\cP]$.
\end{Conjecture}

This extends previous conjectures, which were restricted to simple polyominoes \cite[Conjecture 3.2]{QRR} or non-simple polyominoes (\cite[Conjecture 4.9]{JN}): here the conjecture is verified computationally for collections of cells up rank $10$, including also polyominoes. Although our experiments provide strong evidence for the validity of the conjecture, they do not amount to a general proof. Nevertheless, as we will show in the next section, the validity of Conjecture \ref{conj} could offer an efficient algebraic method to determine the switching rook polynomial and the rook number for arbitrary collections of cells (see Example \ref{Exa: rook polynomial}).

\section{Quadratic Gröbner basis of collections of cells}\label{sec:3}
Let $<$ be a monomial order on $S_\Pc$. For any $f \in S_{\Pc}$ and any ideal $I \subseteq S_{\Pc}$, we denote by $\mathrm{in}_{<}(f)$ the initial monomial of $f$ and by $\mathrm{in}_{<}(I)$ the initial ideal of $I$ with respect to $<$. Throughout this paper, we consider the following total order on the variables of $S_{\Pc}$: For $x_a, x_b \in S_{\Pc}$ with $a = (i,j)$ and $b = (k,\ell)$, we set $x_a > x_b$ if $i > k$, or if $i = k$ and $j > \ell$. 

Let $<_{\rev}$ and $<_{\lex}$ denote the reverse lexicographical and lexicographical monomial orders on $S_{\Pc}$ induced by this total order, respectively. Note that for any inner interval $[p, q]$ of $\Pc$ with anti-diagonal corners $r$ and $s$, we have
\[
\mathrm{in}_{<_{\rev}}(x_p x_q - x_r x_s) = x_r x_s \quad \text{and} \quad \mathrm{in}_{<_{\lex}}(x_p x_q - x_r x_s) = x_p x_q.
\]

We recall the following result from \cite{Q}:

\begin{Theorem}[{\cite[Theorem 4.1, Remark 4.2]{Q}}]\label{Theorem: Qureshi condition to have Groebner basis}
Let $\Pc$ be a collection of cells. The set of generators of $I_\Pc$ forms a reduced (quadratic) Gr\"obner basis with respect to $<_{\rev}$ if and only if the following condition holds:
\begin{enumerate}
    \item[{\em($\#$)}] For any two inner intervals $[b,a]$ and $[d,c]$ of $\Pc$ with anti-diagonal corners $e, f$ and $f, g$, as shown in Figure~\ref{Figura: conditions for Groebner basis}(A), either the pair $b, g$ or the pair $e, c$ forms the anti-diagonal corners of an inner interval of $\Pc$.
\end{enumerate}

Moreover, the set of generators of $I_\Pc$ forms a reduced (quadratic) Gr\"obner basis with respect to $<_{\lex}$ if and only if the following condition holds:
\begin{enumerate}
    \item[{\em($\#'$)}] For any two inner intervals $[g,f]$ and $[f,e]$ of $\Pc$ with anti-diagonal corners $a,b$ and $c,d$, as shown in Figure~\ref{Figura: conditions for Groebner basis}(B), either the pair $b, g$ or the pair $e, c$ forms the \emph{diagonal} corners of an inner interval of $\Pc$.
\end{enumerate}
\begin{figure}[h]
    \centering
    \subfloat[]{\includegraphics[scale=0.7]{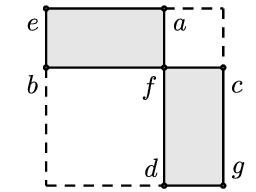}}\qquad
    \subfloat[]{\includegraphics[scale=0.7]{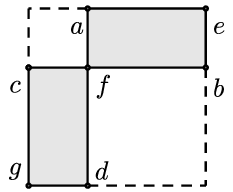}}
    \caption{Conditions for a quadratic Gr\"obner basis with respect to $<_{\rev}$ and $<_{\lex}$.}
    \label{Figura: conditions for Groebner basis}
\end{figure}
\end{Theorem}

\begin{Remark}\label{rem:condtion}
    {\em Let $\Pc$ be a collection of cells and $\Qc$ be a reflection of $\Pc$.
    Since the arrangements of intervals described in Figure~\ref{Figura: conditions for Groebner basis}(A) and Figure~\ref{Figura: conditions for Groebner basis}(B) are reflections of each other, it follows that $\Pc$ satisfies Condition~($\#$) if and only if $\Qc$ satisfies Condition~($\#'$).}
\end{Remark}

\begin{Proposition}\label{prop:QbasisDirected}
    Let $\Pc$ be a directed convex polyomino. Then $I_\Pc$ (up to a suitable symmetry of $\Pc$) has a quadratic Gr\"obner basis with respect to $<_{\rev}$. In particular, $I_\Pc$ (up to a suitable symmetry of $\Pc$) has a quadratic Gr\"obner basis with respect to $<_{\rev}$ when $\Pc$ is a parallelogram polyomino, a stack polyomino, or a Ferrers diagram. 
\end{Proposition}

\begin{proof}
    Let $\Pc$ be a directed convex polyomino with minimal bounding rectangle $[A,Z]$. Following the definition of directed convexity and Remark~\ref{rem:p=q}, we may assume that $A \in \Pc$. We only need to show that $\Pc$ satisfies Condition~($\#$) given in Theorem~\ref{Theorem: Qureshi condition to have Groebner basis}. 
    
    Assume that $[b,a]$ and $[d,c]$ are inner intervals of $\Pc$ with anti-diagonal corners $e, f$ and $f, g$, as shown in Figure~\ref{Figura: conditions for Groebner basis}(A). We claim that $b$ and $g$ are anti-diagonal corners of an inner interval of $\Pc$.

    Let $C \in \Pc$ be the cell with top-right corner $c$, and let $B, D \in \Pc$ be the cells with bottom-left corners $b$ and $d$, respectively. Since $\Pc$ is a polyomino, there exists a path of cells in $\Pc$ from $A$ to $C$. Then any such path must contain a cell that is either in a vertical position with $B$ or in a horizontal position with $D$. In both cases, due to the convexity of $\Pc$, we immediately obtain an inner interval of $\Pc$ with anti-diagonal corners $b$ and $g$. This completes the proof of the claim, and therefore $\Pc$ satisfies Condition~($\#$). 
\end{proof}

\begin{Remark}\label{rem:revandlex}
  Let $\Pc$ be a directed convex polyomino with minimal bounding rectangle $[A,B]$, and let $A'$ and $B'$ be the anti-diagonal corners of $[A,B]$. Following the proof of Proposition~\ref{prop:QbasisDirected} and using similar arguments, we have the following:
  
  \begin{enumerate}
      \item If $A \in \Pc$ or $B \in \Pc$, then $I_\Pc$ has a quadratic Gr\"obner basis with respect to $<_{\rev}$.
      \item If $A' \in \Pc$ or $B' \in \Pc$, then $I_\Pc$ has a quadratic Gr\"obner basis with respect to $<_{\lex}$.
  \end{enumerate}
\end{Remark}

We recall the following result from \cite{Q}, which provides insight into the structure of weakly connected convex collections of cells.

\begin{Lemma}\cite[Lemma 1.3]{Q}\label{lemma:weakly1}
Let $\Pc$ be a simple collection of cells, and let $\Pc_1$ and $\Pc_2$ be two connected components of $V(\Pc)$. Then $|V(\Pc_1) \cap V(\Pc_2)|\leq 1$.
\end{Lemma}

Now we present another class of collection of cells that admit quadratic Gröbner basis. 

\begin{Proposition}\label{prop:orientation}
Let $\Pc$ be a weakly connected convex collection of cells with more than one connected component. Then, up to a suitable symmetry of $\Pc$, the ideal $I_\Pc$ has a quadratic Gröbner basis with respect to $<_{\rev}$.
\end{Proposition}
\begin{proof}
Let $\Pc$ be a weakly connected convex collection of cells with connected components $\Pc_1, \ldots, \Pc_r$ where $r \geq 2$. Due to the convexity of $\Pc$, for any $i \neq j$, none of the cells in $\Pc_i$ is in horizontal or vertical position with any cell in $\Pc_j$. Since every convex polyomino is also a simple polyomino, Lemma~\ref{lemma:weakly1} yields 
$|V(\Pc_i) \cap V(\Pc_j)| \leq 1$ for all $i \neq j$. Therefore, after renumbering the indices, we may assume that $|V(\Pc_i) \cap V(\Pc_{i+1})| = 1$, and for all $i \neq j$ with $j \notin \{i-1, i+1\}$, we have $V(\Pc_i) \cap V(\Pc_j) = \emptyset$.

We first observe that $\Pc_1$ and $\Pc_r$ are directed convex, and $\Pc_i$ is a parallelogram for all $i \in \{2, \dots, r-1\}$. To see this, for each $i \in [r-1]$, let $C_i \in \Pc_i$ and $D_{i+1} \in \Pc_{i+1}$ be such that $|V(C_i) \cap V(D_{i+1})| = 1$. Moreover, for each $i \in [r]$, let $[A_i, B_i]$ denote the minimal bounding rectangle of $\Pc_i$. The convexity of $\Pc$ implies that for each $i \in [r-1]$, the cells $C_i$ and $D_{i+1}$ are among the corner cells of $[A_i, B_i]$. Thus, $\Pc$ takes one of the forms shown in Figure~\ref{fig: Weakly connected convex collection of cells.} (A) and (B).

Since the reflection of Figure~\ref{fig: Weakly connected convex collection of cells.} (B) along the $y$-axis yields a collection of the form in Figure~\ref{fig: Weakly connected convex collection of cells.} (A), we may assume, by applying Remark~\ref{rem:p=q}, that $\Pc$ is of the form in Figure~\ref{fig: Weakly connected convex collection of cells.} (A). In other words, we assume that $|V(A_{i+1}) \cap V(B_i)| = 1$ for all $i \in [r-1]$.

\begin{figure}[h]
    \centering
    \subfloat[]{\includegraphics[scale=0.5]{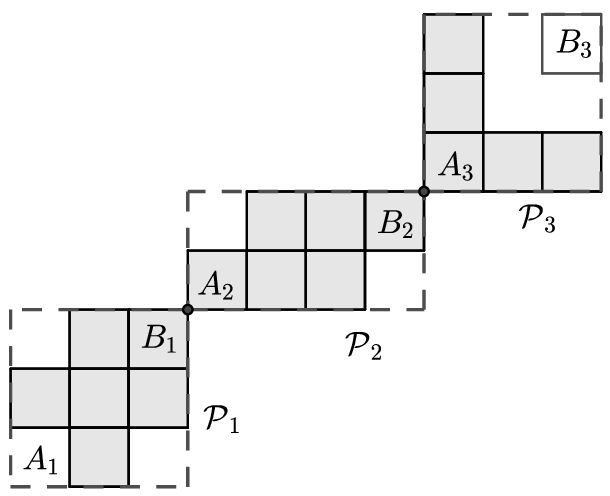}}\qquad
    \subfloat[]{\includegraphics[scale=0.5]{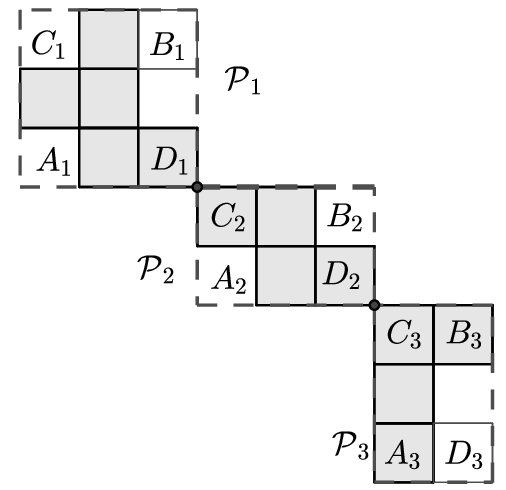}}\qquad
    \subfloat[]{\includegraphics[scale=0.5]{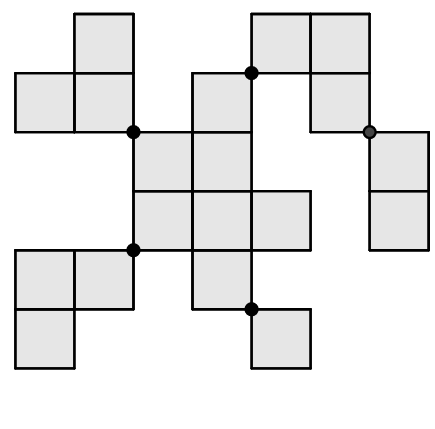}}
    \caption{Two weakly connected convex collections of cells and a weakly connected, simple and non-convex collection of cells, whose connected components are convex.}
    \label{fig: Weakly connected convex collection of cells.}
\end{figure}

To complete the proof, by Theorem~\ref{Theorem: Qureshi condition to have Groebner basis}, it is enough to show that $\Pc$ satisfies Condition~($\#$). This follows by observing that if $[b,a]$ and $[d,c]$ are inner intervals of $\Pc$ with anti-diagonal corners $e,f$ and $f,g$ as shown in Figure~\ref{Figura: conditions for Groebner basis} (A), then $[b,a]$ and $[d,c]$ are inner intervals of some $\Pc_i$ for $1 \leq i \leq r$. It then follows from item (1) of Remark~\ref{prop:QbasisDirected} that $\Pc_1, \ldots, \Pc_r$ each have a quadratic Gröbner basis with respect to $<_{\rev}$. In particular, each of $\Pc_1, \ldots, \Pc_r$ satisfies Condition~($\#$).
\end{proof}

Observe that the assumption of convexity is essential in the above result. For example, the weakly connected collection of cells in Figure~\ref{fig: Weakly connected convex collection of cells.} (C) does not satisfy either Condition~($\#$) or Condition~($\#'$).

% \begin{figure}[h]
%     \centering
%     \includegraphics[scale=0.5]{Figure_eps/Convex_collection_cells_3.eps}
%     \caption{A weakly connected, simple, and non-convex collection of cells, whose connected components are convex.}
%     \label{fig:nonconvex}
% \end{figure}

%%%%%%%%%%%%%%%%%%  section   %%%%%%%%%%%%%%%%%
%%%%%%%%%%%%%%%%%%%%%%%%%%%%%%%%%%%%%%%%%%%%%%%%%%%%%%%%%%%%%%%%%%%%%%%%%%%%%%%%%%%%%%%%%%%%%%%%%%%%%%%%%%

\section{Hilbert series of convex collection of cells with quadratic Gr\"obner basis}\label{Sec:4}

Our goal is to prove Conjecture~\ref{conj:into} in the case of convex collections of cells with quadratic Gröbner basis. More precisely, we prove the following:

\begin{Theorem}\label{Thm:main}
    Let $\cP$ be a collection of cells whose weakly connected components are convex. If $\cP$ satisfies either Condition~($\#$) or Condition~($\#'$), then $h_{K[\cP]}(t)$ is the switching rook polynomial of $\cP$. Moreover, $\mathrm{reg}(K[\cP])$ is equal to the rook number of $\cP$.
\end{Theorem}

As a consequence, together with \cite[Corollary 2.3]{CNJ}, we obtain the following result.

\begin{Corollary}\label{Coro of main}
    Let $\cP$ be a simple collection of cells whose connected components are convex polyominoes satisfying Condition~($\#$) or Condition~($\#'$). Then $h_{K[\cP]}(t)$ is the switching rook polynomial of $\cP$ and $\mathrm{reg}(K[\cP])$ is equal to the rook number of $\cP$.
\end{Corollary}

To prove the above theorem, we need some preparation. First, we find a suitable symmetry of $\cP$ which allows us to dissect $\Pc$ in a useful way. We begin with the following assumption.

\begin{Assumption}\label{assump:1}
Due to Remark~\ref{rem:tensorprod}, it is enough to consider the case when $\Pc$ has exactly one weakly connected component. Furthermore, due to Remarks~\ref{rem:p=q} and \ref{rem:condtion}, to prove Theorem~\ref{Thm:main}, it is enough to consider the case when $\cP$ is a weakly connected convex collection of cells satisfying Condition~($\#$). Following the proof of Proposition~\ref{prop:orientation}, we have that $\Pc$ satisfies Condition~($\#$) if and only if it is of the form shown in Figure~\ref{fig: Weakly connected convex collection of cells.}(A). Therefore, in the sequel, we assume that $\Pc$ is of the form given in Figure~\ref{fig: Weakly connected convex collection of cells.}(A) and satisfies Condition~($\#$). 
\end{Assumption}

Working under Assumption~\ref{assump:1}, we introduce the following setup, which will be used throughout the following text.

\begin{Setup}\label{subsec:setup}
Let $\Pc$ be a weakly connected convex collection of cells as illustrated in Figure~\ref{fig: Weakly connected convex collection of cells.}(A). Let $C_1, \ldots, C_m$ be the columns of $\Pc$ numbered from left to right, and let $R_1, \ldots, R_n$ be the rows of $\Pc$ numbered from bottom to top. Let $k$ be the minimum integer such that $C_k \cap R_n\neq \emptyset$. We refer to the cell $X$ in this intersection as the {\em top left cell} of $\Pc$, and let $v$ be the upper left corner of $X$. Note that if $v$ belongs to an inner interval $I$ of $\cP$, then $v$ is the left anti-diagonal corner of $I$. Let $\Pc_v$ be the union of all inner intervals of $\Pc$ containing $v$. Let $[X,Z]$ and $[Y,X]$ be the maximal horizontal and vertical cell intervals of $\Pc$, respectively. Then the minimal bounding rectangle of $\Pc_v$ is $[Y,Z]$. Note that $\Pc_v$ is a Ferrers diagram because the three corners of $[Y,Z]$, namely $X$, $Y$, and $Z$, belong to $\Pc_v$. Let $Y \in R_j$ and $Z\in C_p$ for some $j$ and $p$. We illustrate this setup in Figure~\ref{fig:Ferrer diagram Pv + LRD + LRDE}(A). The weakly connected convex collections of cells considered in Figure~\ref{fig:Ferrer diagram Pv + LRD + LRDE} have only one connected component; that is, they are polyominoes. 

\begin{figure}[h]
    \centering
    \subfloat[]{\includegraphics[scale=0.6]{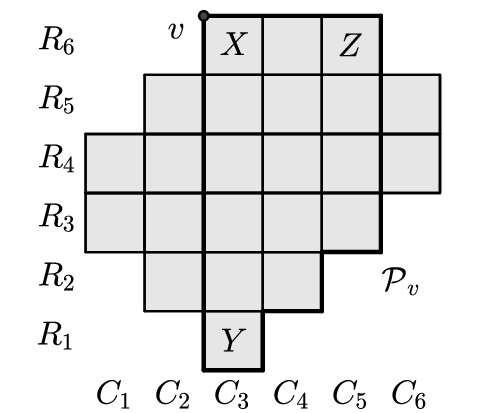}}\quad
    \subfloat[]{\includegraphics[scale=0.6]{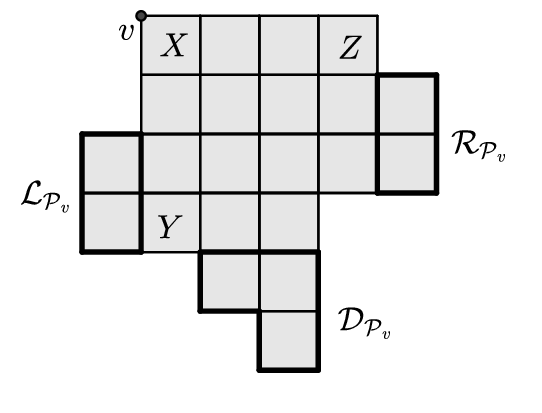}}
    \subfloat[]{\includegraphics[scale=0.6]{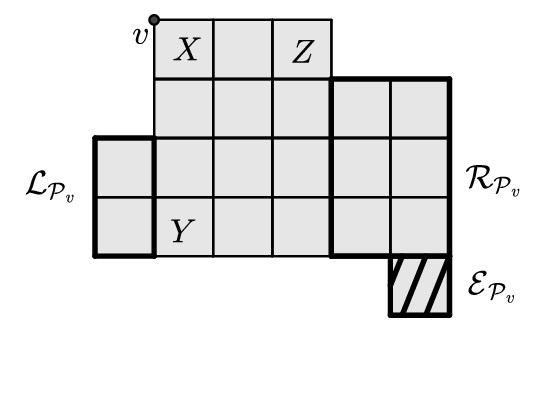}}
    \caption{An illustration of the setup}
    \label{fig:Ferrer diagram Pv + LRD + LRDE}
\end{figure}

We make the following observations:
\begin{enumerate}
    \item If $j > 1$, then $C_k \cap R_{\ell} = \emptyset$ for all $1 \leq \ell \leq j-1$, because $[Y,X]$ is a maximal vertical cell interval of $\Pc$. 
    \item If $p < m$, then $C_{\ell} \cap R_n = \emptyset$ for all $p+1 \leq \ell \leq m$, because $[X,Z]$ is a maximal horizontal cell interval of $\Pc$.
\end{enumerate}

We distinguish the following subcollections of cells of $\Pc$; see Figures~\ref{fig:Ferrer diagram Pv + LRD + LRDE}(B) and (C) for reference. Simply speaking, these subcollections represent the cells of $\Pc$ that appear to the left, right, below, and diagonally right-below of the Ferrers diagram $\Pc_v$.
\begin{enumerate}
    \item[{\em(1)}] $\Lc_{\Pc_v}= \{C \in \Pc : C \in C_q \text{ and } 1 \leq q \leq k-1\}$.
    
    \item[{\em(2)}] $\Rc_{\Pc_v}= \{C \in \Pc : C \in C_q \cap R_{q'} \text{ and } p+1 \leq q \leq m,\ j \leq q' \leq n-1\}$.
    
    \item[{\em(3)}] $\Dc_{\Pc_v}= \{C \in \Pc : C \in C_q \cap R_{q'} \text{ and } k+1 \leq q \leq p,\ 1 \leq q' \leq j-1\}$.
    
    \item[{\em(4)}] $\Ec_{\Pc_v}= \{C \in \Pc : C \in C_q \cap R_{q'} \text{ and } p+1 \leq q \leq m,\ 1 \leq q' \leq j-1\}$.
\end{enumerate}

Note that $\Ec_{\Pc_v} \neq \emptyset$ only if $[Y,Z] \subseteq \Pc$. Moreover, $\Lc_{\Pc_v}$ is disjoint from $\Rc_{\Pc_v}$, $\Dc_{\Pc_v}$, and $\Ec_{\Pc_v}$. 

\end{Setup}

\begin{figure}[h]
    \centering
    \includegraphics[scale=0.55]{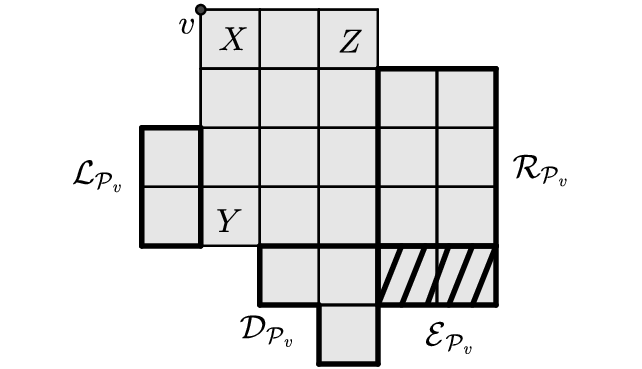}
    \caption{Convex polyomino with $\Ec_{\Pc_v} \neq \emptyset$ but not satisfying Condition~($\#$). }
    \label{fig:Ferrer diagram LRDE no condition}
\end{figure}

With the setup introduced above, we proceed to identify a suitable symmetry of $\Pc$ in the following proposition. For the polyomino $\Pc$ depicted in Figure~\ref{fig:Ferrer diagram LRDE no condition}, we observe that $\Ec_{\Pc_v} \neq \emptyset$, which implies that the subcollections $\Rc_{\Pc_v}$ and $\Dc_{\Pc_v}$ intersect at a vertex. It is easy to see that this polyomino does not satisfy Condition~($\#$). The proposition below clarifies this situation.

%%%%%%%%%%%%%%%%%%%%%%%%%%%%%%%%%%%%%%%%%%%%%%%%%%%%

\begin{Proposition}\label{prop:setup}
Let $\Pc$ be a weakly connected convex collection of cells that satisfies Condition~($\#$). Then:
\begin{enumerate}
    \item[{\em(i)}] If $\Ec_{\Pc_v} \neq \emptyset$, then exactly one of $\Rc_{\Pc_v}$ or $\Dc_{\Pc_v}$ is nonempty, and the other must be empty.
    \item[{\em(ii)}] If $\Ec_{\Pc_v} = \emptyset$, then, up to a suitable symmetry of $\Pc$, we may assume that both $\Ec_{\Pc_v} = \emptyset$ and $\Dc_{\Pc_v} = \emptyset$.
\end{enumerate}
\end{Proposition}

\begin{proof}
Assume $\Pc$ is a weakly connected convex collection of cells satisfying Condition~($\#$), and let its connected components be $\Pc_1,\dots,\Pc_r$, with $|V(\Pc_i)\cap V(\Pc_{i+1})|=1$ for each $i$. Recall the configuration in Figure~\ref{fig:Ferrer diagram Pv + LRD + LRDE}(A).

If $r \ge 2$, let $v \in V(\Pc_r)$ and observe that all other components lie in $\Lc_{\Pc_v}$. Hence $\Dc_{\Pc_v} = \Ec_{\Pc_v} = \emptyset$, proving (ii).

Now consider $r = 1$, so $\Pc$ is a convex polyomino. Let us apply the Setup from \ref{subsec:setup}.

\medskip
\noindent\textbf{(i)} Suppose $\Ec_{\Pc_v} \neq \emptyset$, and take any cell $E \in \Ec_{\Pc_v}$. We claim that exactly one of $\Rc_{\Pc_v}$, $\Dc_{\Pc_v}$ is nonempty. Indeed, if both were empty, Condition~($\#$) would fail. If both were nonempty, pick $B \in \Rc_{\Pc_v}$ and $D \in \Dc_{\Pc_v}$. Let $D$ lie in column $C_a$ for some $k+1\leq a \leq p$ and $B$ lie in row $R_b$ for some $j\leq b\leq n-1$. Since $\Ec_{\Pc_v}\neq \emptyset$, we obtain $j>1$ and $p < n$. Observe that $C_{p+1} \cap R_{j-1} = \emptyset$. Indeed, if $C_{p+1} \cap R_{j-1} = C$, then the cell interval $[Y,Z]$ and $C$ violate the condition ($\#$), a contradiction. Since $\Pc$ is convex, it follows immediately that $C_{p+1}\cap R_{b} \neq \emptyset$ and $C_{a}\cap R_{j-1} \neq \emptyset$. Since $C_{p+1} \cap R_{j-1} = \emptyset$, due to the convexity of $\Pc$, we obtain $C_{p+1} \cap R_{q'} = \emptyset$ for all $1 \leq q' \leq j-1$, and $C_q \cap R_{j-1} = \emptyset$ for all $p+1 \leq q \leq m$. Therefore, the cells in $\Ec_{\Pc_v}$ cannot be connected to the cells in $\Pc \setminus \Ec_v$, a contradiction to the definition of weakly connected collection of cells. Thus exactly one of the subcollections is nonempty, proving (i).
 
% pick $B \in \Rc_{\Pc_v}$ and $D \in \Dc_{\Pc_v}$. Let $B$ lie in column $C_a$ (with $p+1 \le a \le m$) and $D$ lie in row $R_b$ (with $1 \le b \le j-1$). Due to $\Ec_{\Pc_v} \neq \emptyset$, we have $j>1$ and $p < m$. Yet $C_{p+1} \cap R_{j-1} = \emptyset$ by convexity and Condition~($\#$), so $B$ and $D$ cannot be weakly connected to the rest—a contradiction. Thus exactly one of the subcollections is nonempty, proving (i).

\medskip
\noindent\textbf{(ii)}  Now assume $\Ec_{\Pc_v} = \emptyset$. Let $\ell$ be the maximum column index such that $C_\ell \cap R_1 \neq \emptyset$.  Let $[Y, W]$ be the maximal horizontal cell interval containing $Y$. Then $W \in C_a$ for some $k \leq a \leq p$. The convexity of $\Pc$ together with the assumption that $\Ec_{\Pc_v}= \emptyset$ and the observation that $C_{k} \cap R_{j-1} = \emptyset$ shows that $k+1 \leq \ell \leq p$.  Denote by $X'$ the cell in $C_\ell \cap R_1$ {\em bottom right cell} of $\Pc$ and the vertex $ u=(\ell+1,1) \in V(\Pc)$ as {\em bottom right corner of $\Pc$}. Let $\Pc_u$ be the subpolyomino of $\Pc$ consisting of all the inner intervals of $\Pc$ containing $u$. (The same way we constructed $\Pc_v$). Let $[A, X']$ be the maximal horizontal cell interval containing $X'$.  Since $C_{k} \cap R_{q'} = \emptyset$ for all $1 \leq q'\leq j-1$, we obtain $A \in C_q$ for some $k+1 \leq q \leq p$. Thus, due to the convexity of $\Pc$, we conclude that $\Pc_u$ is a rectangle $[B, X']$, where $B \in [X, Z]$. We rotate $\Pc$ 180 degrees to obtain $\Qc$ with top left cell $X'$ and top left corner $u$. Since this rotation preserves the Condition~($\#$), and $I_{\Pc}= I_{\Qc}$ as noted in Remark~\ref{rem:p=q}, we may replace $\Pc$ with $\Qc$ in our discussion above to obtain $\Dc_{\Pc_v}=\Ec_{\Pc_v}= \emptyset$. 
\end{proof}

%%%%%%%%%%%%%%%%%%%%%%%%%%%%%%%%%%%%%%%%%%%%%%%%%%%%%%%%%%%%%%%%%%%%%%%%%%%%%%%%%%%%%%%%%%%%%%%%%%%%%
\begin{Assumption}\label{assump:rotation}
In light of Proposition~\ref{prop:setup}, to prove Theorem~\ref{Thm:main}, we may assume that the weakly connected convex collection of cells $\Pc$ satisfies one of the following conditions:
\begin{enumerate}
    \item[(i)] $\Ec_{\Pc_v} = \Dc_{\Pc_v} = \emptyset$.
    \item[(ii)] $\Ec_{\Pc_v} \neq \emptyset$, and exactly one of $\Rc_{\Pc_v}$ or $\Dc_{\Pc_v}$ is nonempty, while the other is empty.
\end{enumerate}
\end{Assumption}

In the sequel, we will work under Assumption~\ref{assump:rotation}. Next, we introduce the following notation:

\begin{Notation}\label{Notation:bar+star}
We define:
\begin{enumerate}
    \item[$\bullet$] If $\Dc_{\Pc_v} = \emptyset$, then $\overline{\Rc_{\Pc_v}} = \Rc_{\Pc_v} \cup \Ec_{\Pc_v}$.
    \item[$\bullet$] If $\Rc_{\Pc_v} = \emptyset$, then $\overline{\Dc_{\Pc_v}} = \Dc_{\Pc_v} \cup \Ec_{\Pc_v}$.
\end{enumerate}

Moreover, if at least one of $\Lc_{\Pc_v}$ or $\overline{\Rc_{\Pc_v}}$ is nonempty, we define $\Lc_{\Pc_v} * \overline{\Rc_{\Pc_v}}$ to be the polyomino obtained by identifying each $(p+1, q) \in V(\overline{\Rc_{\Pc_v}})$ with $(k, q)$, as illustrated in Figure~\ref{fig: gluing of LRDE}. In other words, we shift each $(i, j) \in V(\overline{\Rc_{\Pc_v}})$ backward by $p+1 - k$ units so that $\Lc_{\Pc_v}$ and $\overline{\Rc_{\Pc_v}}$ can be glued together. In the case where $\Lc_{\Pc_v} = \overline{\Rc_{\Pc_v}} = \emptyset$, we set $\Lc_{\Pc_v} * \overline{\Rc_{\Pc_v}} = \emptyset$.

Note that $\Lc_{\Pc_v} * \overline{\Rc_{\Pc_v}}$ is also a weakly connected convex collection of cells. Furthermore:
\begin{itemize}
    \item If $\overline{\Rc_{\Pc_v}} = \emptyset$, then $\Lc_{\Pc_v} * \overline{\Rc_{\Pc_v}} = \Lc_{\Pc_v}$;
    \item If $\Lc_{\Pc_v} = \emptyset$, then $\Lc_{\Pc_v} * \overline{\Rc_{\Pc_v}}$ is simply a translation of $\overline{\Rc_{\Pc_v}}$.
\end{itemize}
\end{Notation}

\begin{figure}[h]
    \centering
    \subfloat[Here $\Dc_{\Pc_v} = \emptyset$]{\includegraphics[scale=0.55]{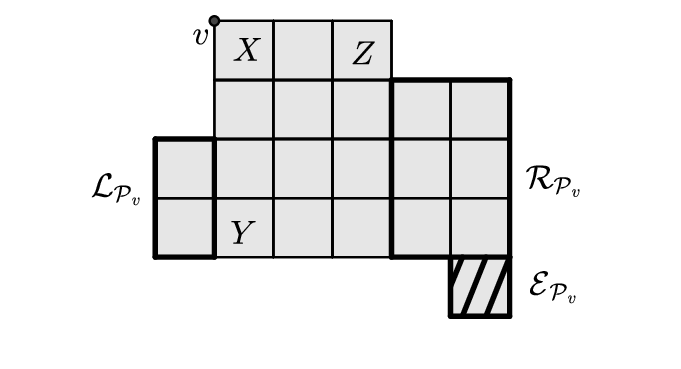}}\quad
    \subfloat[$\Lc_{\Pc_v} * \overline{\Rc_{\Pc_v}}$]{\includegraphics[scale=0.55]{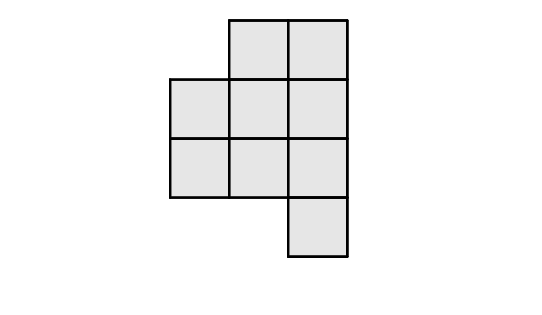}}
    \caption{An example of $\Lc_{\Pc_v} * \overline{\Rc_{\Pc_v}}$.}
    \label{fig: gluing of LRDE}
\end{figure}

We now recall some basic definitions in graph theory. Let $G$ be a finite graph with vertex set $V(G)$ and edge set $E(G)$. If $\{u, v\}\in E(G)$, then $u$ and $v$ are said to be adjacent in $G$. For a vertex $v \in V(G)$, the \textit{open neighborhood} of $v$, denoted by $N_G(v)$, is the set of vertices of $G$ adjacent to $v$. The \textit{closed neighborhood} of $v$, denoted by $N_G[v]$, is defined as $N_G(v)\cup \{v\}$. For a subset $A\subset V(G)$, the graph $G \setminus A$ is obtained by removing all vertices in $A$ and all edges incident to a vertex in $A$.

Let $G$ be a graph with vertex set $\{x_1, \dots, x_n\}$, and let $R=K[x_1, \dots, x_n]$ be a polynomial ring over a field $K$. For convenience, we use $x_i$ both to denote a vertex of $G$ and as a variable in $R$. The \textit{edge ideal} $I(G)$ associated with $G$ is the ideal in $R$ generated by all squarefree monomials $x_i x_j$ such that $x_i$ is adjacent to $x_j$ in $G$.

Let $\Pc$ be a convex collection of cells satisfying condition~$(\#)$. It follows from Proposition~\ref{Theorem: Qureshi condition to have Groebner basis} that the initial ideal of $I_\Pc$ with respect to the term order $<_{\rev}$, denoted by $\ini_{<_{\rev}}(I_{\Pc})$, is generated by monomials $x_{a}x_{b}$ where $a$ and $b$ are the antidiagonal corners of some inner interval of $\Pc$. Since $\ini_{<_{\rev}}(I_{\Pc})$ is generated in degree 2, we may view it as the edge ideal of a graph. In the sequel, we denote by $G_{\Pc}$ the graph whose edge ideal is $\ini_{<_{\rev}}(I_{\Pc})$.

\begin{Lemma}\label{Lemma: Polyomino P'}
    Let $X$ be the top-left corner cell and let $v$ be the top-left corner of a convex collection of cells $\Pc$ satisfying condition~$(\#)$. Let $\Pc' = \Pc \setminus \{X\}$. Then $\Pc'$ also satisfies condition~$(\#)$, and $G_{\Pc} \setminus \{v\} = G_{\Pc'}$.
\end{Lemma}

\begin{proof}
 The statement follows from our choice of $v$.
\end{proof}

%%%%%%%%%%%%%%%%%%%%%%%%%%%%%%%%

\begin{Proposition}\label{Prop: Polyomino P''}
   Let $X$ be the top-left corner cell and $v$ be the top-left corner of a convex collection of cells $\Pc$ such that $\Pc$ satisfies condition~$(\#)$. We define a collection of cells $\Pc''$ as follows:
   \[
   \Pc''=\begin{cases} 
      \Lc_{\Pc_v} \cup \overline{\Dc_{\Pc_v}}, & \text{if } \Dc_{\Pc_v} \neq \emptyset, \\
      \Lc_{\Pc_v} * \overline{\Rc_{\Pc_v}}, & \text{otherwise}.
   \end{cases}
   \]
   Then $\Pc''$ also satisfies condition~$(\#)$, and $G_{\Pc} \setminus N[v] \cong G_{\Pc''}$.
\end{Proposition}

\begin{proof}
We first show that $\Pc''$ satisfies condition~$(\#)$. Assume that in $\Pc''$ there exist two inner intervals $[b,a]$ and $[d,c]$ as described in Figure~\ref{Figura: conditions for Groebner basis}(A).

Consider first the case $\Pc'' = \Lc_{\Pc_v} \cup \overline{\Dc_{\Pc_v}}$. Since $\Lc_{\Pc_v} \cap \overline{\Dc_{\Pc_v}} = \emptyset$ (see the last line in Setup~\ref{subsec:setup}), we have either $[b,a], [d,c] \subset \Lc_{\Pc_v}$ or $[b,a], [d,c] \subset \overline{\Dc_{\Pc_v}}$.

Suppose $[b,a], [d,c] \subset \Lc_{\Pc_v}$. Since $\Pc$ satisfies Condition~$(\#)$, it follows that either $b$ and $g$ or $e$ and $c$ are antidiagonal corners of an inner interval of $\Pc$, as in Figure~\ref{Figura: conditions for Groebner basis}(A). Suppose $b$ and $g$ are such corners of an inner interval $I$ in $\Pc$. By the construction of $\Lc_{\Pc_v}$, we have $I \subseteq \Lc_{\Pc_v}$. A similar argument applies if $[b,a], [d,c] \subset \overline{\Dc_{\Pc_v}}$.

Now consider the case $\Pc'' = \Lc_{\Pc_v} * \overline{\Rc_{\Pc_v}}$. The desired conclusion follows by an analogous argument, using the observation that $\Lc_{\Pc_v} * \overline{\Rc_{\Pc_v}}$ is simply obtained by removing $\Pc_v$ from $\Pc$.

The isomorphism $G_{\Pc} \setminus N[v] \cong G_{\Pc''}$ follows directly from the construction of $\Pc''$.
\end{proof}

%%%%%%%%%%%%%%%%%%%%

 %{\color{red} here we remark that if we had considered top right cell and corner, we won't have obtained above result. }
%%%%%%%%%%%%%%%%%%%%%%%%%%%%%%%%%%%%%%%%%%%%%%%%%%%%%%%%%%%%

\begin{Remark} \label{Remark: Isolated vertices of the graph}
    Let $\cP$ be a weakly connected convex collection of cells that satisfies the Condition~($\#$).  Let $v$ be the top-left corner of $\cP$ and $X$ be the top-left corner cell of $\cP$, as in Setup~\ref{subsec:setup}. Denote by $H$ and $V$ the horizontal and vertical maximal edge intervals of $\cP$ containing $v$, respectively. 
    From the results discussed so far, we can establish a natural identification of the vertices $V(\Dc_{\Pc_v})\cap V(\cP_v)$ (resp. $V(\Rc_{\Pc_v})\cap V(\cP_v)$)  with some of vertices of $H$ (resp. $V$). Specifically:
    \begin{itemize}
        \item for $\Dc_{\Pc_v}\neq \emptyset$, we identify each vertex $(i,j) \in V(\Dc_{\Pc_v})\cap V(\cP_v)$ with $(i,n)\in H$. In this case, we define $\pi(\Dc_{\Pc_v})$ as the set of vertices of $H$ corresponding to this identification. For convenience, we set $\pi(\Dc_{\Pc_v})=\emptyset$ if $\Dc_{\Pc_v}=\emptyset$.
        \item for $\Rc_{\Pc_v}\neq \emptyset$, we identify each vertex $(p+1,j) \in V(\Rc_{\Pc_v})\cap V(\cP_v)$ with $(k,j)\in V$. In this case, we define $\pi(\Rc_{\Pc_v})$ as the set of the vertices of $V$ corresponding to this identification. For convenience, we set $\pi(\Rc_{\Pc_v})=\emptyset$ if $\Rc_{\Pc_v}= \emptyset$. We set $\Gamma(\Rc_{\Pc_v})=\pi(\Rc_{\Pc_v})\cup\left(V(\Lc_{\Pc_v})\cap V(\cP_v)\right)$.
    \end{itemize}
Due to Proposition \ref{prop:setup}, we have $\pi(\Dc_{\Pc_v})\cap\Gamma(\Rc_{\Pc_v})=\emptyset$. Define the integer $d=\vert H\cup V\vert - \vert \pi(\Dc_{\Pc_v})\vert-\vert \Gamma(\Rc_{\Pc_v})\vert.$  
\end{Remark}

The notation introduced in Remark~\ref{Remark: Isolated vertices of the graph} is useful for the following result.

\begin{Proposition}\label{Proposition: decomposition Hilbert series}
    Let $\cP$ be a weakly connected convex collection of cells such that $\Pc$ satisfies Condition~{\em($\#$)}. Let $v$ be the top-left corner of $\cP$ and $X$ be the top-left corner cell of $\cP$. Then
    \[
    \rHP_{K[\cP]}(t) = \rHP_{K[\cP']}(t) + \frac{t}{(1 - t)^{d}} \rHP_{K[\cP'']}(t),
    \]    
    where $\cP'=\cP\setminus\{X\}$, $\cP''$ is the convex collection of cells defined in Proposition~\ref{Prop: Polyomino P''}, and $d$ is as defined in Remark~\ref{Remark: Isolated vertices of the graph}. 
\end{Proposition}

\begin{proof}
     Since $\rHP_{K[\cP]}(t)=\rHP_{S_\cP/\ini_{<_{\rev}}(I_\cP)}(t)$, see for example \cite[Corollary 6.1.5]{HH_Book}, it is enough to study the Hilbert-Poincar\'{e} series of $S_\cP/\ini_{<_{\rev}}(I_\cP)$. Let $v$ be the top-left corner of $\cP$ and consider the following exact sequence: 
     \[
\begin{tikzcd}
  0 \arrow[r] &   S_{\cP}/(\ini_{<_{\rev}}(I_\cP):x_v)  \arrow[r, "x_v"] & S_{\cP}/\ini_{<_{\rev}}(I_\cP)  \arrow[r] &  S_{\cP}/(\ini_{<_{\rev}}(I_\cP),x_v)  \arrow[r] & 0 
\end{tikzcd}
\]

This exact sequence gives 
\[
\rHP_{S_{\cP}/\ini_{<_{\rev}}(I_\cP)}(t)=\rHP_{S_{\cP}/(\ini_{<_{\rev}}(I_\cP),x_v)}(t)+t \; \rHP_{S_{\cP}/(\ini_{<_{\rev}}(I_\cP):x_v)}(t).
\]
We now analyze $S_{\cP}/(\ini_{<_{\rev}}(I_\cP),x_v)$ and $S_{\cP}/(\ini_{<_{\rev}}(I_\cP):x_v)$:
     \begin{enumerate}
         \item From Lemma~\ref{Lemma: Polyomino P'}, it follows that $S_{\cP}/(\ini_{<_{\rev}}(I_\cP),x_v)\cong S_{\cP'}/\ini_{<_{\rev}}(I_{\cP'})=K[\cP']$. Hence 
         \[
         \rHP_{S_{\cP}/(\ini_{<_{\rev}}(I_\cP),x_v)}(t)=\rHP_{K[\cP']}(t).
         \]
        \item By Proposition~\ref{Prop: Polyomino P''} and the identity $\ini_{<_{\rev}}(I_\cP):x_v=\ini_{<_{\rev}}(I_{\cP''})+(x_u:u\in V(\cP_v)\setminus (H\cup V))$, we obtain
        \[
            S_{\cP}/(\ini_{<_{\rev}}(I_\cP):x_v)\cong K[\cP'']\tensor_K K[x_u:u\in (H\cup V)\setminus (\pi(\cD_\cP)\sqcup\Gamma(\cB_\cP))].
        \]
        Since the Hilbert-Poincar\'{e} series is multiplicative under tensor products over $K$, it follows that
        \[
        \rHP_{S_{\cP}/(\ini_{<_{\rev}}(I_\cP):x_v)}(t)=\frac{1}{(1 - t)^{d}} \rHP_{K[\cP'']}(t).
        \]       
     \end{enumerate}
     Thus, the claim follows directly from the previous two points.
\end{proof}

Now we present the proof of Theorem~\ref{Thm:main}.

\begin{proof}[Proof of Theorem~\ref{Thm:main}]
We proceed by induction on the number of cells in $\cP$. If $|\cP| = 1$, then $I_\cP$ is trivially the determinantal ideal of a $2 \times 2$ matrix. In this case, $h_{K[\cP]}(t) = 1 + t$, which coincides with the switching rook polynomial of $\cP$.

Assume now that $|\cP| > 1$, and suppose that the claim holds for every collection of cells $\cF$ with $1 \leq |\cF| < |\cP|$ satisfying Condition~$(\#)$ or Condition~$(\#')$, and such that each weakly connected component of $\cF$ is convex. That is, $h_{K[\cF]}(t)$ equals the switching rook polynomial of $\cF$.

We aim to show that $h_{K[\cP]}(t)$ is the switching rook polynomial of $\cP$. If $\cP$ has more than one weakly connected component, then the assertion follows from the inductive hypothesis and Remark~\ref{rem:tensorprod}, using the multiplicativity of both the Hilbert series and the switching rook polynomial under disjoint unions.

Now assume that $\cP$ consists of a single weakly connected component. We use the notation and assumptions introduced in Setup~\ref{subsec:setup}, Assumption~\ref{assump:rotation}, and Notation~\ref{Notation:bar+star}.

By Proposition~\ref{Proposition: decomposition Hilbert series}, and using the notation defined therein, we have
\begin{equation}\label{Equation: Hilbert series}
    \rHP_{K[\cP]}(t) = \rHP_{K[\cP']}(t) + \frac{t}{(1 - t)^{d}} \rHP_{K[\cP'']}(t).
\end{equation}

Denote by $\tilde{r}_{\cP'}(t) = \sum_{j=0}^{r(\cP')} \tilde{r}_j(\cP') t^j$ and $\tilde{r}_{\cP''}(t) = \sum_{j=0}^{r(\cP'')} \tilde{r}_j(\cP'') t^j$ the switching rook polynomials of $\cP'$ and $\cP''$, respectively. Note that $\cP'$ and $\cP''$ are collections of cells with convex weakly connected components and satisfy $|\cP'|, |\cP''| < |\cP|$. 

Moreover, by Lemma~\ref{Lemma: Polyomino P'} and Proposition~\ref{Prop: Polyomino P''}, both $\cP'$ and $\cP''$ satisfy Condition~($\#$). Therefore, by the inductive hypothesis, we have $h_{K[\cP']}(t) = \tilde{r}_{\cP'}(t)$ and $h_{K[\cP'']}(t) = \tilde{r}_{\cP''}(t)$.

Since $\cP$, $\cP'$, and $\cP''$ are simple collections of cells, by \cite[Proposition~1.3]{CNJ}, we have:
\[
\dim K[\cP] = |V(\cP)| - |\cP|, \quad \dim K[\cP'] = |V(\cP')| - |\cP'|, \quad \text{and} \quad \dim K[\cP''] = |V(\cP'')| - |\cP''|.
\]

Using Equation~\eqref{Equation: Hilbert series}, we can rewrite the Hilbert series as:
\begin{equation}\label{Equation: decomposition h-polynomial into rook polynomials}
    \rHP_{K[\cP]}(t) = \frac{h_{K[\cP]}(t)}{(1 - t)^{|V(\cP)| - |\cP|}} 
    = \frac{\tilde{r}_{\cP'}(t)}{(1 - t)^{|V(\cP')| - |\cP'|}} 
    + \frac{t \tilde{r}_{\cP''}(t)}{(1 - t)^{|V(\cP'')| - |\cP''| + d}}.
\end{equation}
Note that $|V(\cP')| = |V(\cP)| - 1$ and $|\cP'| = |\cP| - 1$, which implies that 
\[
|V(\cP')| - |\cP'| = |V(\cP)| - |\cP| = \dim K[\cP].
\]

We now prove that 
\[
|V(\cP'')| - |\cP''| + d = |V(\cP)| - |\cP|.
\]

Indeed, we have:
\[
|V(\cP'')| = |V(\cP)| - |V(\cP_v)| + \pi(\Dc_{\Pc_v}) + |\Gamma(\Rc_{\Pc_v})|, \quad \text{and} \quad
|\cP''| = |\cP| - |\cP_v|.
\]

Observe that there is a one-to-one correspondence between the cells of $\cP_v$ and their lower-right corners, which correspond exactly to $V(\cP_v) \setminus (H \cup V)$, where $H$ and $V$ are the horizontal and vertical maximal edge intervals of $\cP$ containing $v$. Consequently, we obtain $|\cP_v| = |V(\cP_v)| - |H \cup V|,$ which leads to:
\[
|\cP''| = |\cP| - |V(\cP_v)| + |H \cup V|.
\]

Thus, we can compute:
\begin{align*}
|V(\cP'')| - |\cP''| + d 
%&= |V(\cP)| - |V(\cP_v)| + \pi(\Dc_{\Pc_v}) + |\Gamma(\Rc_{\Pc_v})| \\&\quad - \left(|\cP| - |V(\cP_v)| + |H \cup V|\right) + d \\
&= |V(\cP)| - |\cP| - |H \cup V| - \pi(\Dc_{\Pc_v}) - |\Gamma(\Rc_{\Pc_v})| + d \\
&= |V(\cP)| - |\cP| - d + d = |V(\cP)| - |\cP|.
\end{align*}

This confirms the claim. Hence, by Equation~(2), we obtain that $h_{K[\cP]}(t) = \tilde{r}_{\cP'}(t) + t\tilde{r}_{\cP''}(t).$ Now, we write 
\[
 \tilde{r}(t):= \tilde{r}_{\cP'}(t) + t\tilde{r}_{\cP''}(t) = \sum_{i=0}^{r(\cP)} \left( \tilde{r}_i(\cP') + \tilde{r}_{i-1}(\cP'') \right) t^i ,
\]
where we set $\tilde{r}_i(\cP') = 0$ for $i \geq r(\cP')$ and $\tilde{r}_i(\cP'') = 0$ for $i \geq r(\cP'')$.

We prove that $\tilde{r}(t)$ is the switching rook polynomial of $\cP$. Fix $i \in \{0, \dots, r(\cP)\}$ and we show that the $i$-th coefficient of $\tilde{r}(t)$ represents the number of $i$-rook configurations of $\cP$, up to switches. To establish this, observe the following:

\begin{enumerate}
    \item $\tilde{r}_i(\cP')$ is the cardinality of $\tilde{\cR}_{i}(\cP')$, which can be viewed as the number of $i$-rook configurations in $\cP$, up to switches, such that no rook is placed in the top-left corner cell $X$.
    
    \item Assume that a rook $T_1$ is placed on $X$, and let $\cC$ be an $(i-1)$-rook configuration in $\cP$. If a rook of $\cC$, say $T_2$, lies on a cell of $\cP_v$, then $T_1$ and $T_2$ occupy anti-diagonal cells of some inner interval $\cP_I$ of $\cP_v$. This implies that $T_1$ and $T_2$ are in switching position. Hence, the configuration 
    \[
    \cC' := (\cC \setminus \{T_1, T_2\}) \cup \{T_1', T_2'\},
    \]
    where $T_1'$ and $T_2'$ are the rooks placed in the diagonal cells of $\cP_I$, belongs to $\cR_i(\cP')$. However, such a configuration has already been counted in point (1). Therefore, we only need to count the $(i-1)$-rook configurations in $\cP$, up to switches, where all rooks lie in $\cP \setminus \cP_v$. Accordingly, the number of $i$-rook configurations in $\cP$, up to switches, where a rook is placed on $X$, coincides with the number of $(i-1)$-rook configurations in $\cP''$, up to switches, that is, $\tilde{r}_{i-1}(\cP'')$.
\end{enumerate}
From the previous two points, we easily deduce that the $i$-th coefficient of $\tilde{r}(t)$ is the number of $i$-rook configurations of $\cP$, up to switches. Hence $h_{K[\cP]}(t)$ is the switching rook polynomial of $\cP$.\\

To conclude the theorem, it only remains to show that $\mathrm{reg}({K[\cP]})$ is the rook number of $\cP$. This is a consequence of the result which we have just obtained. In fact, we know that, if $\cP$ is a simple collection of cells, then $K[\cP]$ is a Cohen-Macaulay domain by \cite[Proposition 1.3]{CNJ}, so $\reg{K[\cP]}$ is the degree of the $h$-polynomial of $K[\cP]$ by \cite[Corollary B.4.1]{Va}, which is the rook number of $\cP$.
\end{proof}

\begin{Example}\rm\label{Exa: rook polynomial}
By Corollary \ref{Coro of main}, and through computations carried out with \texttt{Macaulay2} (\cite{M2}), we obtain that for the convex polyominoes $\cP_1$ and $\cP_2$, illustrated in Figure \ref{fig:Ferrer diagram and stack} (D) and Figure \ref{Figura:example rook configuration} (A) respectively, and for the simple collection of cells $\cP_3$ depicted in Figure \ref{fig: Weakly connected convex collection of cells.} (C), the rook numbers are $r(\cP_1)=6$, $r(\cP_2)=5$, and $r(\cP_3)=11$. Moreover, the corresponding switching rook polynomials are 
\[
\begin{aligned}
\tilde{r}_{\cP_1}(t) = h_{K[\cP_1]}(t) &= 1 + 29t + 215t^{2} + 533t^{3} + 459t^{4} + 121t^{5} + 7t^{6}, \\
\tilde{r}_{\cP_2}(t) = h_{K[\cP_2]}(t) &= 1 + 23t + 139t^{2} + 302t^{3} + 257t^{4} + 81t^{5}+ 7t^{6}, \\
\tilde{r}_{\cP_3}(t) = h_{K[\cP_3]}(t) &= 1 + 19t + 152t^2 + 676t^3 + 1851t^4 + 3259t^5 + 3742t^6 +\\
&\quad + 2788t^7 + 1319t^8 + 379t^9 + 60t^{10} + 4t^{11}.
\end{aligned}
\]

\noindent By using a standard machine and the proposed algorithm (\cite{N}), the computation of the rook number and  the switching rook polynomial for $\cP_2$ takes approximately 20 minutes, instead for $\cP_1$ and $\cP_3$ requires excessive resources, making it practically infeasible within a reasonable time. In contrast, the computation of the $h$-polynomial of the associated coordinate ring is extremely efficient by \texttt{Macaulay2}, being completed almost instantaneously (about $\sim 1$ second) in all three cases.
\end{Example}

%%%%%%%%%%%%%%%%%%%%%%%%%%%%%%%%%%%%%%%%%%%%%%%%%%%%%%%%%%%%%%%%%%%%%%%%
%%%%%%%%%%%%%%%%%%%%%%%%%%%%%%%%%%%%%%%%%%%%%%%%%%%%%%%%%%%%%%%%%%%%%%%%
%%%%%%%%%%%%%%%%%%%%%%%%%%%%%%%%%%%%%%%%%%%%%%%%%%%%%%%%%%%%%%%%%%%%%%%%

%\footnotesize{{\bf Declaration of competing interest.}  The authors declare that they have no known competing financial interests or personal relationships that could have appeared to influence the work reported in this paper.}

%\footnotesize{{\bf Data availability.} All the data are provided in the repository \cite{N}.}

%{\bf Data availability}No data was used for the research described in the article.

%%%%%%%%%%%%%%%%%%%%%%%%%%%%%%%%%%%%%%%%%%%%%%%%%%%%%%%%%%%%%%%%%%%%%%%%%%%%%%%%%%%%%%%%%%%%
%%%%%%%%%%%%%%%%%%%%%%%%%%%%%%%%%%%%%%%%%%%%%%%%%%%%%%%%%%%%%%%%%%%%%%%%%%%%%%%%%%%%%%%%%%%%

	\end{document}